\documentclass{amsart}
\usepackage{euler, eufrak}
\usepackage[all]{xy}

\theoremstyle{plain}
\newtheorem{theorem}{Theorem}[subsection]
\newtheorem*{theorem*}{Theorem}
\newtheorem{proposition}[theorem]{Proposition}
\newtheorem{lemma}[theorem]{Lemma}

\theoremstyle{remark}

\theoremstyle{definition}

\let\a\alpha
\let\b\beta
\let\g\gamma
\let\d\delta
\let\e\varepsilon
\let\z\zeta
\let\th\theta
\let\k\kappa

\let\s\sigma

\let\y\upsilon
\let\f\varphi

\def\lra{\longrightarrow}
\def\egal{\ar@{=}}
\def\surj{\ar@{->>}}

\def\E{\mathcal E}
\def\F{\mathcal F}
\def\G{\mathcal G}

\def\K{\mathcal K}
\def\O{\mathcal O}
\def\Y{\mathcal Y}
\def\Z{\mathcal Z}

\def\Ker{{\mathcal Ker}}
\def\Coker{{\mathcal Coker}}
\def\Im{{\mathcal Im}}

\def\AA{\mathbb A}
\def\CC{\mathbb C}
\def\PP{\mathbb P}

\def\WW{\mathbb W}
\def\ZZ{\mathbb Z}

\def\D{{\scriptscriptstyle \operatorname{D}}}
\def\TR{{\scriptscriptstyle \operatorname{T}}}

\def\H{\operatorname{H}}
\def\h{\operatorname{h}}
\def\M{\operatorname{M}}
\def\N{\operatorname{N}}

\def\Sym{\operatorname{S}}
\def\T{\operatorname{T}}

\def\Hom{\operatorname{Hom}}
\def\Aut{\operatorname{Aut}}

\def\Ext{\operatorname{Ext}}
\def\GL{\operatorname{GL}}

\def\Stab{\operatorname{Stab}}

\def\tensor{\otimes}
\def\isom{\simeq}

\def\ba{\begin{array}}
\def\ea{\end{array}}

\title[Homology of the space of sheaves of with Hilbert polynomial $5m+3$]
{On the homology of the moduli space of plane sheaves with Hilbert polynomial $\boldsymbol{5m+3}$}

\author{Mario Maican}
\address{Institute of Mathematics of the Romanian Academy,
Calea Grivitei 21, Bucharest 010702, Romania}
\email{mario.maican@imar.ro}

\keywords{Semi-stable sheaves, Bia{\l}ynicki-Birula decomposition, Betti numbers, Hodge numbers}
\subjclass[2010]{14D22.}

\begin{document}

\maketitle

\begin{abstract}
We compute the Hodge numbers of the moduli space of semi-stable sheaves on the complex projective
plane supported on quintic curves and having Euler characteristic $3$.
For this purpose we study the fixed-point set for a certain torus action on the moduli space.
\end{abstract}

\section{Introduction}  
\label{section_1}

Let $\M_{\PP^2}(r, \chi)$ denote the moduli space of semi-stable sheaves on $\PP^2 = \PP^2(\CC)$ having
Hilbert polynomial $P(m) = rm+ \chi$. Our goal is to determine the additive structure of the singular homology groups
with coefficients in $\ZZ$ for $\M = \M_{\PP^2}(5,3)$, which, according to \cite{lepotier}, is a smooth projective variety of dimension $26$.
We refer to the introductory section of \cite{choi_maican} for an overview of the present state of research into the geometry
of the moduli spaces $\M_{\PP^2}(r, \chi)$.
The Poincar\'e polynomial of $\M_{\PP^2}(5,3)$ has already been computed
in \cite{yuan} by means of a cellular decomposition.
We will use, instead, the Bia\l{y}nicki-Birula method \cite{birula, birula_polonici, carrell},
which has the advantage of yielding the Hodge numbers, as well.
This method consists of determining the $T$-fixed locus
and the $T$-representation of the tangent spaces at the fixed points for the action of a torus $T$ on a smooth projective variety.
We refer to \cite[Section 2]{choi_maican} for a brief outline of the Bia\l{y}nicki-Birula method.
In \cite[Sections 5, 6]{choi_maican} this method was used to study the homology of $\M_{\PP^2}(4,1)$,
relying on a stratification by strata that are easily understood as geometric quotients.
We will apply the same technique to $\M_{\PP^2}(5,3)$.
We will use the stratification of this moduli space provided in \cite{illinois}, which we recall at the beginning of Section \ref{section_2}.
The action of $T$ on each stratum is easy to study because, as mentioned, the strata are geometric quotients.
More challenging is the problem of determining the $T$-representation of the normal spaces to the strata,
which we solve at Propositions \ref{3.3.1} and \ref{3.4.1}.
Denote
\[
T = (\CC^*)^3 / \{ (c, c, c),\ c \in \CC^* \}.
\]
Let $T$ act on $\PP^2$ by
\[
(t_0, t_1, t_2) \cdot (x_0, x_1, x_2) = (t_0^{-1} x_0^{}, \ t_1^{-1} x_1^{}, \ t_2^{-1} x_2^{}).
\]
Denote by $\mu_t \colon \PP^2 \to \PP^2$ the map of multiplication by $t \in T$.
Our main result below concerns the action of $T$ on $\M_{\PP^2}(5,3)$ given by
\[
t [\F] = [\mu_{t^{-1}}^* \F],
\]
where $[\F]$ denotes the stable-equivalence class of a sheaf $\F$.

\begin{theorem*} 
The fixed point locus of $\M_{\PP^2}(5,3)$ consists of $1329$ isolated points,
$174$ projective lines, and three isomorphic surfaces obtained by blowing up $\PP^1 \times \PP^1$
at three points on the diagonal, then blowing down the strict transform of the diagonal.
The integral homology groups of $\M_{\PP^2}(5,3)$ have no torsion and its Poincar\'e polynomial is
\begin{align*}
P(x) = & x^{52} + 2x^{50}+6x^{48}+13x^{46}+26x^{44}+45x^{42}+68x^{40}+87x^{38}+100x^{36} \\
& +107x^{34}+111x^{32}+112x^{30}+113x^{28}+113x^{26}+113x^{24}+112x^{22}+111x^{20} \\
& +107x^{18}+100x^{16}+87x^{14}+68x^{12}+45x^{10}+26x^8+13x^6+6x^4+2x^2+1.
\end{align*}
The Euler characteristic of $\M_{\PP^2}(5,3)$ is $1695$ and its Hodge numbers satisfy the relations
\[
h^{pq}=0 \quad \text{if} \quad p \neq q \quad \text{and} \quad h^{pp} = b_{2p},
\]
where $b_{2p}$ are the Betti numbers from above.
\end{theorem*}

\noindent
Our calculation of the Poincar\'e polynomial agrees thus with \cite{yuan}.
Intriguingly, this is the same as the Poincar\'e polynomial of $\M_{\PP^2}(5,1)$,
computed in \cite{yuan} and \cite{choi_chung}.
This raises the question whether $\M_{\PP^2}(5,1)$ and $\M_{\PP^2}(5,3)$ are (canonically) isomorphic.
Such an isomorphism would imply that there is only one smooth moduli space of semi-stable sheaves supported
on plane quintics. Indeed, by duality \cite{rendiconti}, $\M_{\PP^2} (r, \chi) \isom \M_{\PP^2}(r, -\chi)$.
(Thus, what we say in this paper about $\M_{\PP^2}(5,3)$ is equally valid for $\M_{\PP^2}(5,-3)$, $\M_{\PP^2}(5,2)$,
etc.)


\section{The torus fixed locus}
\label{section_2}

For the convenience of the reader we recall from \cite{illinois} the classification of
semi-stable sheaves on $\PP^2$ having Hilbert polynomial $5m+3$.
In $\M_{\PP^2}(5,3)$ we have an open stratum $\M_0$, two locally closed strata $\M_1$
and $\M_2$ of codimension $2$, $3$, and a closed stratum $\M_3$ of codimension $4$.
Each $\M_{i+1}$ is contained in the closure of $\M_i$.
The open stratum consists of sheaves $\F$ having a presentation of the form
\[
0 \lra 2\O(-2) \oplus \O(-1) \stackrel{\f}{\lra} 3\O \lra \F \lra 0,
\]
where $\f$ is not equivalent to a morphism represented by a matrix having one of the following forms:
\[
\left[
\ba{ccc}
\star & \star & \star \\
\star & \star & 0 \\
\star & \star & 0
\ea
\right] \qquad \text{or} \qquad \left[
\ba{ccc}
\star & \star & \star \\
\star & \star & \star \\
\star & 0 & 0
\ea
\right].
\]
We denote by $\M_{01}$ the closed subset of $\M_0$ given by the condition that the entries of $\f_{12}$
span a vector space of dimension $2$ in $V^*$.
Furthermore, inside $\M_{01}$ we distinguish the open subset (in the relative topology)
$\M_{010}$ of cokernels of morphisms
of the form
\[
\left[
\ba{ccc}
\star & \star & l_1 \\
\star & \star & l_2 \\
q_1 & q_2 & 0
\ea
\right],
\]
where $l_1, l_2 \in V^*$ are linearly independent and $q_1$, $q_2$ have no common factor.
The complement $\M_{011} = \M_{01} \setminus \M_{010}$ is given by the condition that $q_1$ and $q_2$
have a common linear factor.

The sheaves giving points in $\M_1$ are precisely those sheaves $\F$ having a resolution of the form
\[
0 \lra 2\O(-2) \oplus 2\O(-1) \stackrel{\f}{\lra} \O(-1) \oplus 3\O \lra \F \lra 0,
\]
where $\f_{12}=0$, $\f_{11}$ has linearly independent entries and $\f_{22}$ has linearly independent
maximal minors. Denote by $\M_{10} \subset \M_1$ the open subset (in the relative topology)
given by the condition that the maximal
minors of $\f_{22}$ have no common factor. The complement $\M_{11} = \M_1 \setminus \M_{10}$
is given by the condition that the maximal minors of $\f_{22}$ have a common linear factor.

The points $[\F]$ in $\M_2$ are given by exact sequences of the form
\[
0 \lra 3\O(-2) \stackrel{\f}{\lra} 2\O(-1) \oplus \O(1) \lra \F \lra 0,
\]
where $\f_{11}$ has linearly independent maximal minors.
Denote by $\M_{20}$ the open subset (in the relative topology) given by the condition that the maximal minors
of $\f_{11}$ have no common factor. The sheaves giving points in $\M_{20}$ are of the form $\O_Q(-\Z)^{\D} (-1)$, where
$\O_Q(-\Z) \subset \O_Q$ is the ideal sheaf of a zero-dimensional scheme $\Z$ of length $3$ contained in a quintic
curve $Q$, where $\Z$ is not contained in a line. We write $\O_Q(\Z)(1) = \O_Q(-\Z)^{\D} (-1)$.
The ideal of $\Z$ is generated by the maximal minors of $\f_{11}$ and $Q$ is given by the equation $\det(\f)=0$.
Denote $\M_{21} = \M_2 \setminus \M_{20}$.
The sheaves $\F$ giving points in $\M_{21}$ are precisely the extension sheaves
\[
0 \lra \O_C(1) \lra \F \lra \O_L \lra 0
\]
satisfying $\H^1(\F)=0$. Here $C$ is a quartic curve and $L$ is a line.
In fact, $L$ is given by the equation $l = 0$, where $l$ is the common divisor of the maximal minors of $\f_{11}$,
and $C$ is given by the equation $\det(\f)/l=0$.

Finally, the sheaves $\F$ in the closed stratum are cokernels of the form
\[
0 \lra \O(-3) \oplus \O(-1) \stackrel{\f}{\lra} \O \oplus \O(1) \lra \F \lra 0,
\]
where $\f_{12}$ is non-zero and does not divide $\f_{22}$. Equivalently, the sheaves giving points in $\M_3$
are of the form $\O_Q(-\Y)(2)$, where $\O_Q(-\Y) \subset \O_Q$ is the ideal sheaf of a zero-dimensional scheme $\Y$
of length $2$ contained in a quintic curve $Q$. Here $\Y$ is given by the ideal $(\f_{12}, \f_{22})$ and $Q$
is the zero-set of $\det(\f)$.

We denote by $W_i$ the set of morphisms $\f$ as above whose cokernels give points in $\M_i$.
The ambient vector space containing $W_i$ is denoted by $\WW_i$.
For instance
\[
\WW_0 = \Hom (2\O(-2) \oplus \O(-1), 3\O).
\]
We denote by $G_0$, $G_1$, $G_2$, $G_3$ the obvious algebraic groups acting by conjugation on $\WW_i$.
For instance
\[
G_0 = (\Aut(2\O(-2) \oplus \O(-1)) \times \Aut(3\O))/\CC^*,
\]
where $\CC^*$ is embedded as the subgroup of homotheties.
According to \cite{illinois}, each $\M_i$ is a geometric quotient of $W_i$ by $G_i$.
In particular, the fibres of the canonical map $W_i \to \M_i$ are precisely the $G_i$-orbits.


\subsection{Fixed points in $\M_0 \setminus \M_{01}$}
\label{2.1}

Consider a $T$-fixed sheaf $\F$ that is the cokernel of a morphism
\[
\f = \left[
\ba{ccc}
q_{11} & q_{12} & X \\
q_{21} & q_{22} & Y \\
q_{31} & q_{32} & Z
\ea
\right]. \quad \text{For all $t \in T$ we have} \quad t \f = \left[
\ba{ccc}
t q_{11} & t q_{12} & t_0 X \\
t q_{21} & t q_{22} & t_1 Y \\
t q_{31} & t q_{32} & t_2 Z
\ea
\right].
\]
Performing, possibly, column operations on $\f$, we may assume that
$q_{11}$ and $q_{12}$ are in $\CC[Y, Z]$.
Since the fibres of the map $W_0 \to \M_0$ are the $G_0$-orbits, we deduce 
that there is $(g(t),h(t)) \in G_0$ such that $t\f = h(t) \f g(t)$.
Clearly, we may assume that
\[
g(t) = \left[
\ba{ccc}
g_{11} & g_{12} & 0 \\
g_{21} & g_{22} & 0 \\
u_1 & u_2 & 1
\ea
\right], \quad \quad h(t) = \left[
\ba{ccc}
t_0 & 0 & 0 \\
0 & t_1 & 0 \\
0 & 0 & t_2
\ea
\right].
\]
From the relations
\begin{align*}
t q_{11} & = t_0 (q_{11} g_{11} + q_{12} g_{21} + X u_1), \\
t q_{12} & = t_0 (q_{11} g_{12} + q_{12} g_{22} + X u_2)
\end{align*}
we deduce that $u_1 = 0$ and $u_2=0$.
Writing
\[
q_{11} = a_1 Y^2 + b_1 Z^2 + c_1 YZ, \qquad q_{12} = a_2 Y^2 + b_2 Z^2 + c_2 YZ,
\]
the above relations are equivalent to the equation
\[
\left[
\ba{ccc}
t_1^2 & 0 & 0 \\
0 & t_2^2 & 0 \\
0 & 0 & t_1 t_2
\ea
\right] \left[
\ba{cc}
a_1 & a_2 \\
b_1 & b_2 \\
c_1 & c_2
\ea
\right] = t_0 \left[
\ba{cc}
a_1 & a_2 \\
b_1 & b_2 \\
c_1 & c_2
\ea
\right] \left[
\ba{cc}
g_{11} & g_{12} \\
g_{21} & g_{22}
\ea
\right].
\]


\subsubsection{Case when $q_{11}$, $q_{12}$ are linearly independent modulo $X$}
\label{2.1.1}

We may assume a priori that one of the matrices
\[
\left[
\ba{cc}
a_1 & a_2 \\
b_1 & b_2
\ea
\right], \qquad \left[
\ba{cc}
a_1 & a_2 \\
c_1 & c_2
\ea
\right], \qquad \left[
\ba{cc}
b_1 & b_2 \\
c_1 & c_2
\ea
\right].
\]
is the identity matrix. In the first case we have
\[
\left[
\ba{cc}
g_{11} & g_{12} \\
g_{21} & g_{22}
\ea
\right] = \left[
\ba{cc}
t_0^{-1} t_1^2 & 0 \\
0 & t_0^{-1} t_2^2
\ea
\right].
\]
For any $t \in T$ we have $t q_{21} = t_0^{-1} t_1^3 q_{21}$. This shows that $q_{21}=0$.
Analogously $q_{22}=0$, $q_{31}=0$, $q_{32}=0$, hence $\det(\f)=0$, which contradicts our
choice of $\f$. In the other two cases we arrive at the same contradiction.
We conclude that there are no $T$-fixed points in $\M_0 \setminus \M_{01}$ for which
$q_{11}$ and $q_{12}$ are linearly independent modulo $X$.
Analogously, there are no fixed points for which $q_{21}$, $q_{22}$ are linearly independent
modulo $Y$ or for which $q_{31}$, $q_{32}$ are linearly independent modulo $Z$.


\subsubsection{Case when $q_{11}$, $q_{12}$ are linearly dependent but not both zero modulo $X$}
\label{2.1.2}

Assume that $q_{12}=0$. For each $t \in T$, $t q_{11} = t_0 q_{11} g_{11}(t)$, hence $q_{11}$ is a monomial,
say $Y^j Z^k$, and $g_{11}(t) = t_0^{-1} t_1^j t_2^k$.
Since $0= t q_{12} = t_0 q_{11} g_{12}$, we get $g_{12}=0$.
Thus $t q_{22} = t_1 q_{22} g_{22}$ and $t q_{32} = t_2 q_{32} g_{22}$, hence $q_{22}$, $q_{32}$
are monomials or zero. We have the relations
\begin{align*}
t q_{21} & = t_0^{-1} t_1^{j+1} t_2^k q_{21} + t_1 q_{22} g_{21}, \\
t q_{31} & = t_0^{-1} t_1^j t_2^{k+1} q_{31} + t_2 q_{32} g_{21}.
\end{align*}
Choosing $t$ such that $t_0^{-1} t_1^{j+1} t_2^k$, $t_0^{-1} t_1^j t_2^{k+1}$ are different from
$t_0^2$, $t_1^2$, $t_2^2$, $t_0^{} t_1^{}$, $t_0^{} t_2^{}$, $t_1^{} t_2^{}$, we see that $q_{21}$ is a multiple
of $q_{22}$ if $q_{22} \neq 0$ and $q_{31}$ is a multiple of $q_{32}$ if $q_{32} \neq 0$.
In the first case, performing possibly column operations on $\f$, we may assume a priori that $q_{21} =0$.
From the first relation above we get $g_{21} =0$.
Likewise, in the second case, we may assume that $g_{21}=0$.
For all $t \in T$ we have
\[
t q_{21} = t_0^{-1} t_1^{j+1} t_2^k q_{21}, \qquad t q_{31} = t_0^{-1} t_1^{j} t_2^{k+1} q_{31},
\]
forcing $q_{21} =0$, $q_{31}=0$.
If both $q_{22}$ and $q_{32}$ are non-zero, then
\[
t \frac{q_{22}}{q_{32}} = \frac{t q_{22}}{t q_{32}} =
\frac{t_1 q_{22} g_{22}}{t_2 q_{32} g_{22}} = t_1^{} t_2^{-1} \frac{q_{22}}{q_{32}}, \quad \text{so} \quad
\frac{q_{22}}{q_{32}} \quad \text{is a multiple of} \quad YZ^{-1},
\]
i.e. there are $l \in \{ X, Y, Z \}$ and $a, b \in \CC^*$ such that $q_{22} = a l Y$, $q_{32} = b l Z$.
We obtain nine isolated $T$-fixed points represented by the matrices
\[
{\mathversion{bold}
\boldsymbol{\a(q_1, q_2)} =
\left[
\ba{ccc}
q_1 & 0 & X \\
0 & q_2 & Y \\
0 & 0 & Z
\ea
\right]}, \quad q_1 \in \{ Y^2, YZ, Z^2 \}, \quad q_2 \in \{ X^2, XZ, Z^2 \},
\]
and eighteen other points if we swap $X$ and $Z$, respectively $Y$ and $Z$.
We obtain nine affine lines of $T$-fixed points represented by the matrices
\[
{\mathversion{bold}
\boldsymbol{\b(q, l)(a,b)} =
\left[
\ba{ccc}
q & 0 & X \\
0 & a lY & Y \\
0 & b lZ & Z
\ea
\right]}, \ q \in \{ Y^2, YZ, Z^2 \}, \ l \in \{X, Y, Z \}, \ (a,b) \in \PP^1 \setminus \{ 1 \},
\]
and eighteen other lines if we swap $X$ and $Y$, respectively $X$ and $Z$.


\subsubsection{Case when $q_{11}$, $q_{12}$ are divisible by $X$,
$q_{21}$, $q_{22}$ are divisible by $Y$, $q_{31}$, $q_{32}$ are divisible by $Z$}
\label{2.1.3}

We may write
\[
\f = \left[
\ba{ccc}
X l_{11} & X l_{12} & X \\
Y l_{21} & Y l_{22} & Y \\
0 & 0 & Z
\ea
\right].
\]
We have the relation
\[
\left[
\ba{ccc}
t_0 X (t l_{11}) & t_0 X (t l_{12}) & t_0 X \\
t_1 Y (t l_{21}) & t_1 Y (t l_{22}) & t_1 Y \\
0 & 0 & t_2 Z
\ea
\right] = \left[
\ba{ccc}
t_0 & 0 & 0 \\
0 & t_1 & 0 \\
0 & 0 & t_2
\ea
\right] \f \left[
\ba{ccc}
g_{11} & g_{12} & 0 \\
g_{21} & g_{22} & 0 \\
u_1 & u_2 & 1
\ea
\right].
\]
From the relations $0 = t_2 Z u_1$ and $0 = t_2 Z u_2$ we see that $u_1=0$, $u_2=0$.
Thus, the above is equivalent to the relation
\[
\left[
\ba{cc}
t l_{11} & t l_{12} \\
t l_{21} & t l_{22}
\ea
\right] = \left[
\ba{cc}
l_{11} & l_{12} \\
l_{21} & l _{22}
\ea
\right] \left[
\ba{cc}
g_{11} & g_{12} \\
g_{21} & g_{22}
\ea
\right].
\]
Write
\[
l_{11} = a_1 X + b_1 Y + c_1 Z, \qquad l_{12} = a_2 X + b_2 Y + c_2 Z,
\]
From the above relation we get the equation
\[
\left[
\ba{ccc}
t_0 & 0 & 0 \\
0 & t_1 & 0 \\
0 & 0 & t_2
\ea
\right] \left[
\ba{cc}
a_1 & a_2 \\
b_1 & b_2 \\
c_1 & c_2
\ea
\right] = t_0 \left[
\ba{cc}
a_1 & a_2 \\
b_1 & b_2 \\
c_1 & c_2
\ea
\right] \left[
\ba{cc}
g_{11} & g_{12} \\
g_{21} & g_{22}
\ea
\right].
\]
Assume that $l_{11}$ and $l_{12}$ are linearly independent.
Performing possibly column operations on $\f$,
we may assume a priori that one of the matrices
\[
\left[
\ba{cc}
a_1 & a_2 \\
b_1 & b_2
\ea
\right], \qquad \left[
\ba{cc}
a_1 & a_2 \\
c_1 & c_2
\ea
\right], \qquad \left[
\ba{cc}
b_1 & b_2 \\
c_1 & c_2
\ea
\right].
\]
is the identity matrix. In the first case we have
\[
\left[
\ba{cc}
g_{11} & g_{12} \\
g_{21} & g_{22}
\ea
\right] = \left[
\ba{cc}
t_0 & 0 \\
0 & t_1
\ea
\right].
\]
For any $t_0, t_1, t_2 \in \CC^*$ we have $t_2 c_1 = t_0 c_1$ and $t_2 c_2 = t_1 c_2$.
Thus $c_1 =0$, $c_1=0$, $l_{11}=X$, $l_{12}=Y$.
For all $t \in T$ we have $t l_{21} = t_0 l_{21}$, $t l_{22} = t_1 l_{22}$,
hence $l_{21} = a X$, $l_{22} = b Y$, where $a , b \in \CC$.
We obtain the fixed points represented by the matrices
\[
{\mathversion{bold}
\boldsymbol{\g(a,b)}= \left[
\ba{ccc}
X^2 & XY & X \\
aXY & b Y^2 & Y \\
0 & 0 & Z
\ea
\right], \qquad a, b \in \CC, \quad a \neq b.}
\]
For the other two cases we obtain the matrices
\[
\left[
\ba{ccc}
X^2 & XZ & X \\
a XY & b YZ & Y \\
0 & 0 & Z
\ea
\right], \quad \text{respectively} \quad \left[
\ba{ccc}
XY & XZ & X \\
a Y^2 & b YZ & Y \\
0 & 0 & Z
\ea
\right].
\]
Note that $\g(a,b)$ is defined for distinct $a, b \in \PP^1$, if we set
\[
\g(\infty,b)= \left[
\ba{ccc}
0 & XY & X \\
XY & b Y^2 & Y \\
0 & 0 & Z
\ea
\right], \quad \text{respectively} \quad \g(a,\infty)= \left[
\ba{ccc}
X^2 & 0 & X \\
aXY & Y^2 & Y \\
0 & 0 & Z
\ea
\right].
\]
Thus we get a surface $\g(a, b)$, $a, b \in \PP^1$, $a \neq b$, of $T$-fixed points in $\M_{\PP^2}(5,3)$
isomorphic to the complement of the diagonal in $\PP^1 \times \PP^1$.
We get two other surfaces of fixed points if we swap $X$ and $Z$, respectively if we swap $Y$ and $Z$.
Note that we have also covered the situation when $l_{21}$ and $l_{22}$ are linearly independent.
It remains to examine the situation when $l_{11}$, $l_{12}$ are linearly dependent and, likewise,
$l_{21}$, $l_{22}$ are linearly dependent. If we discount $\g(0, \infty)$, $\g(\infty, 0)$ and four other points
obtained by interchanging $X$ and $Z$, respectively by interchanging $Y$ and $Z$, leaves us with three
isolated $T$-fixed points represented by the matrices
\[
{\mathversion{bold}
\boldsymbol{\a(q_1, q_2)} = \left[
\ba{ccc}
q_1 & 0 & X \\
0 & q_2 & Y \\
0 & 0 & Z
\ea
\right], \quad (q_1, q_2) \in \{ (X^2, XY), \ (XY, Y^2), \ (XZ, YZ) \}.}
\]


\subsection{Fixed points in $\M_{01}$}
\label{2.2}

Assume that the point in $\M_{\PP^2}(5,3)$ represented by the morphism
\[
\f = \left[
\ba{ccc}
q_{11} & q_{12} & l_1 \\
q_{21} & q_{22} & l_2 \\
q_1 & q_2 & 0
\ea
\right]
\]
is fixed by $T$. Then, as noted before, for each $t \in T$, there is $(g(t), h(t)) \in G_0$ such that $t \f = h(t) \f g(t)$.
We may write
\[
g(t) = \left[
\ba{ccc}
g_{11} & g_{12} & 0 \\
g_{21} & g_{22} & 0 \\
u_1 & u_2 & 1
\ea
\right], \qquad h(t) = \left[
\ba{ccc}
h_{11} & h_{12} & h_{13} \\
h_{21} & h_{22} & h_{23} \\
h_{31} & h_{32} & h_{33}
\ea
\right].
\]
From the relation $0 = h_{31} l_1 + h_{32} l_2$ we get $h_{31} = 0$, $h_{32}=0$.
We have the relation
\[
\left[
\ba{c}
t l_1 \\
t l_2
\ea
\right] = \left[
\ba{cc}
h_{11} & h_{12} \\
h_{21} & h_{22}
\ea
\right] \left[
\ba{c}
l_1 \\
l_2
\ea
\right].
\]
We can argue as at \ref{2.1.3} to deduce that $l_1$ and $l_2$ are distinct elements in the set $\{ X, Y, Z \}$
and that
\[
\left[
\ba{cc}
h_{11} & h_{12} \\
h_{21} & h_{22}
\ea
\right] = \left[
\ba{cc}
(t l_1)/l_1 & 0 \\
0 & (t l_2)/l_2
\ea
\right].
\]
We have the relation
\[
\left[
\ba{cc}
t q_1 & t q_2
\ea
\right] = h_{33} \left[
\ba{cc}
q_1 & q_2
\ea
\right] \left[
\ba{cc}
g_{11} & g_{12} \\
g_{21} & g_{22}
\ea
\right].
\]
Arguing as at \ref{2.1.1}, we can show that $q_1$ and $q_2$ are distinct elements in the set
$\{ X^2, Y^2, Z^2, XY, XZ, YZ \}$ and that
\[
\left[
\ba{cc}
g_{11} & g_{12} \\
g_{21} & g_{22}
\ea
\right] = h_{33}^{-1} \left[
\ba{cc}
(t q_1)/q_1 & 0 \\
0 & (t q_2)/q_2
\ea
\right].
\]
One linear form among $l_1$, $l_2$ does not divide $q_1$ or $q_2$, otherwise $q_1$ and $q_2$ would
be both equal to $l _1 l_2$.
Performing, possibly, elementary operations on $\f$, we may assume that $l_1$ does not divide $q_1$,
that $q_{11}$ and $q_{12}$ do not contain any monomial divisible by $l_1$, and that $q_{11}$ and $q_{21}$
do not contain the monomial $q_1$.
From the relation
\[
t q_{11} = \left( \frac{t l_1}{l_1} q_{11} + h_{13} q_1 \right) h_{33}^{-1} \frac{t q_1}{q_1} + (t l_1) u_1
\]
we get $h_{13}=0$, $u_1=0$. From the relation
\[
t q_{12} = \left( \frac{t l_1}{l_1} q_{12} + h_{13} q_2 \right) h_{33}^{-1} \frac{t q_2}{q_2} + (t l_1) u_2
= \frac{t l_1}{l_1} q_{12} h_{33}^{-1} \frac{t q_2}{q_2} + (t l_1) u_2
\]
we get $u_2 = 0$. From the relation
\[
t q_{21} = \left( \frac{t l_2}{l_2} q_{21} + h_{23} q_1 \right) h_{33}^{-1} \frac{t q_1}{q_1} + (t l_2) u_1
= \left( \frac{t l_2}{l_2} q_{21} + h_{23} q_1 \right) h_{33}^{-1} \frac{t q_1}{q_1}
\]
we get $h_{23}=0$. We deduce that the equation $t \f = h(t) \f g(t)$ is equivalent to the relation
\[
\left[
\ba{cc}
t q_{11} & t q_{12} \\
t q_{21} & t q_{22}
\ea
\right] = h_{33}^{-1} \left[
\ba{cc}
(t l_1)/l_1 & 0 \\
0 & (t l_2)/l_2
\ea
\right] \left[
\ba{cc}
q_{11} & q_{12} \\
q_{21} & q_{22}
\ea
\right] \left[
\ba{cc}
(t q_1)/q_1 & 0 \\
0 & (t q_2)/q_2
\ea
\right].
\]
It becomes clear now that $q_{11}$, $q_{12}$, $q_{21}$, $q_{22}$ are monomials or zero and that
$h_{33}(t) = t_0^i t_1^j t_2^k$ for some integers $i$, $j$, $k$ satisfying $i+j+k=1$.
In fact,
\[
\left[
\ba{cc}
q_{11} & q_{12} \\
q_{21} & q_{22}
\ea
\right] = \left[
\ba{cc}
c_{11} X^{-i} Y^{-j} Z^{-k} l_1 q_1 & c_{12} X^{-i} Y^{-j} Z^{-k} l_1 q_2 \\
c_{21} X^{-i} Y^{-j} Z^{-k} l_2 q_1 & c_{22} X^{-i} Y^{-j} Z^{-k} l_2 q_2
\ea
\right],
\]
where $c_{rs} = 0$ if the corresponding monomial has negative exponents.

\noindent \\
Given $l_1$, $l_2$, $q_1$, $q_2$ as above, and a monomial $d$ of degree $5$,
we denote by $\M(l_1, l_2, q_1, q_2)$ the image in $\M_{\PP^2}(5,3)$ of the set of morphisms
of the form
\[
\f = \left[
\ba{ccc}
\star & \star & l_1 \\
\star & \star & l_2 \\
q_1 & q_2 & 0
\ea
\right],
\]
and by $\M(l_1, l_2, q_1, q_2, d)$ the subset given by the additional condition that $\det(\f)$ be a multiple of $d$.
Clearly, this sets are $T$-invariant.

\begin{proposition}
\label{2.2.1}
Assume that $l_1, l_2 \in \{ X, Y, Z \}$ are distinct one-forms and $q_1, q_2 \in \{ X^2, Y^2, Z^2, XY, XZ, YZ \}$
are distinct two-forms. Then, for any monomial $d$ of degree $5$ belonging to the ideal
$(l_1 q_1, l_1 q_2, l_2 q_1, l_2 q_2)$,
the set of fixed points for the action of $T$ on $\M(l_1, l_2, q_1, q_2, d)$ has precisely one irreducible
component, which is either a point or an affine line.
\end{proposition}

\begin{proof}
We will only examine the case when $l_1 = X$, $l_2=Y$, $q_1=Z^2$, $q_2 = XZ$, all other cases being analogous.
Consider, therefore, a morphism of the form
\[
\f = \left[
\ba{ccc}
c_{11} X^{1-i} Y^{-j} Z^{2-k} & c_{12} X^{2-i} Y^{-j} Z^{1-k} & X \\
c_{21} X^{-i} Y^{1-j} Z^{2-k} & c_{22} X^{1-i} Y^{1-j} Z^{1-k} & Y \\
Z^2 & XZ & 0
\ea
\right],
\]
where $i+j+k = 1$. Assume that $c_{12} \neq 0$. Then $i=2$, $c_{11}=0$, $c_{21}=0$, $c_{22}=0$,
so we get three isolated fixed points:
\[
\left[
\ba{ccc}
0 & Y^2 & X \\
0 & 0 & Y \\
Z^2 & XZ & 0
\ea
\right], \qquad 
\left[
\ba{ccc}
0 & YZ & X \\
0 & 0 & Y \\
Z^2 & XZ & 0
\ea
\right], \qquad
\left[
\ba{ccc}
0 & Z^2 & X \\
0 & 0 & Y \\
Z^2 & XZ & 0
\ea
\right].
\]
Assume now that $c_{12} = 0$ and $c_{11} \neq 0$.
Then $i=1$, $c_{21} =0$, $j+k=0$, $j \le 0$, $k \le 2$.
Thus $(j,k)$ is one of the following pairs: $(0,0)$, $(-1,1)$, $(-2,2)$.
The first case is not feasible because of our convention that $q_{11}$ do not contain
the monomial $q_1$. We obtain the fixed points
\[
\f_1(c) = \left[
\ba{ccc}
YZ & 0 & X \\
0 & cY^2 & Y \\
Z^2 & XZ & 0
\ea
\right], \quad c \in \CC \setminus \{ -1 \}, \qquad
\left[
\ba{ccc}
Y^2 & 0 & X \\
0 & 0 & Y \\
Z^2 & XZ & 0
\ea
\right].
\]
Assume next that $c_{11}=0$, $c_{12}=0$, $c_{21} \neq 0$.
Then $i \le 0$, $j \le 1$, $k \le 2$, hence $(i,j,k)$ is one of the following triples:
$(0,1,0)$, $(0,0,1)$, $(-1,1,1)$, $(0,-1,2)$, $(-1,0,2)$, $(-2,1,2)$.
The first case is not feasible because of our convention that $q_{21}$ do not contain
the monomial $q_1$. We obtain the fixed points
\[
\f_2(c) = \left[
\ba{ccc}
0 & 0 & X \\
YZ & cXY & Y \\
Z^2 & XZ & 0
\ea
\right], \qquad
\f_3(c) = \left[
\ba{ccc}
0 & 0 & X \\
XZ & cX^2 & Y \\
Z^2 & XZ & 0
\ea
\right], \quad c \in \CC \setminus \{ 1 \},
\]
\[
\left[
\ba{ccc}
0 & 0 & X \\
Y^2 & 0 & Y \\
Z^2 & XZ & 0
\ea
\right], \qquad \left[
\ba{ccc}
0 & 0 & X \\
XY & 0 & Y \\
Z^2 & XZ & 0
\ea
\right], \qquad \left[
\ba{ccc}
0 & 0 & X \\
X^2 & 0 & Y \\
Z^2 & XZ & 0
\ea
\right].
\]
The final case to examine is when $c_{11}=0$, $c_{12}=0$, $c_{21}=0$, $c_{22} \neq 0$.
We have $i \le 1$, $j \le 1$, $k \le 1$, hence $(i,j,k)$ belongs to the list
$(1,1,-1)$, $(1,0,0)$, $(1,-1,1)$, $(0,1,0)$, $(0,0,1)$, $(-1,1,1)$.
We get the fixed points
\[
\left[
\ba{ccc}
0 & 0 & X \\
0 & Z^2 & Y \\
Z^2 & XZ & 0
\ea
\right], \quad \phantom{\f_2(\infty) =} \left[
\ba{ccc}
0 & 0 & X \\
0 & YZ & Y \\
Z^2 & XZ & 0
\ea
\right], \quad \f_1(\infty) = \left[
\ba{ccc}
0 & 0 & X \\
0 & Y^2 & Y \\
Z^2 & XZ & 0
\ea
\right],
\]
\[
\left[
\ba{ccc}
0 & 0 & X \\
0 & XZ & Y \\
Z^2 & XZ & 0
\ea
\right], \quad \f_2(\infty) = \left[
\ba{ccc}
0 & 0 & X \\
0 & XY & Y \\
Z^2 & XZ & 0
\ea
\right], \quad \f_3(\infty) = \left[
\ba{ccc}
0 & 0 & X \\
0 & X^2 & Y \\
Z^2 & XZ & 0
\ea
\right].
\]
In conclusion, for the action of $T$ on $\M(X, Y, Z^2, XZ)$,
we have nine fixed isolated points and three affine lines, namely
\[
\{ [\f_1(c)], \ c \in \PP^1 \setminus \{ -1 \} \}, \quad
\{ [\f_2(c)], \ c \in \PP^1 \setminus \{ 1 \} \}, \quad
\{ [\f_3(c)], \ c \in \PP^1 \setminus \{ 1 \} \}.
\qedhere
\]
\end{proof}

\noindent
We denote by $\boldsymbol{\d(l_1,l_2,q_1,q_2,d)}$ the set of $T$-fixed points in $\M(l_1,l_2,q_1,q_2,d)$.
The list of fixed affine lines can be found in Table 1 at the end of this section.

\setcounter{subsubsection}{1}


\subsubsection{Limits of sequences of points in $\M_0 \setminus \M_{01}$}
\label{2.2.2}

By an abuse of notation, in the sequel a matrix will denote its image in the moduli space.
We have
\begin{align*}
\lim_{a \to 1} \b(q,l)(a,1) & = \lim_{a \to 1} \left[
\ba{ccc}
q & 0 & X \\
0 & alY & Y \\
0 & lZ  & Z
\ea
\right] = \lim_{a \to 1} \left[
\ba{ccc}
q & -lX & X \\
0 & (a-1) lY & Y \\
0 & 0  & Z
\ea
\right] \\
& = \lim_{c \to 0} \left[
\ba{ccc}
q & -lX & cX \\
0 & lY & Y \\
0 & 0 & Z
\ea
\right] = \left[
\ba{ccc}
q & lX & 0 \\
0 & lY & Y \\
0 & 0 & Z
\ea
\right] \\
& = \d(Y, Z, q, lX, qlYZ).
\end{align*}
Thus, each of the $27$ affine fixed lines in $\M_0 \setminus \M_{01}$ contains a point from $\M_{01}$
in its closure. We obtain $27$ irreducible components of the $T$-fixed locus isomorphic to $\PP^1$.
Denote by $S$ the closure of the set $\{ \g(a,b) \mid a,b \in \PP^1, \ a \neq b \}$.
Assume that $a, b \in \CC$ are distinct. We have
\begin{align*}
\lim_{c \to 0} \g(ca,cb) & = \lim_{c \to 0} \left[
\ba{ccc}
X^2 & XY & X \\
acXY & bcY^2 & Y \\
0 & 0 & Z
\ea
\right] = \lim_{c \to 0} \left[
\ba{ccc}
X^2 & XY & cX \\
aXY & bY^2 & Y \\
0 & 0 & Z
\ea
\right] \\
& = \phantom{\lim_{c \to 0}} \left[
\ba{ccc}
X^2 & XY & 0 \\
aXY & bY^2 & Y \\
0 & 0 & Z
\ea
\right] = \d(Y, Z, X^2, XY, X^2 Y^2 Z)(a,b).
\end{align*}
\begin{align*}
\lim_{c \to 0} \g(ca+1, cb+1) & = \lim_{c \to 0} \left[
\ba{ccc}
X^2 & XY & X \\
(ca+1)XY & (cb+1)Y^2 & Y \\
0 & 0 & Z
\ea
\right] \\
& = \lim_{c \to 0} \left[
\ba{ccc}
0 & 0 & X \\
ca XY & cb Y^2 & Y \\
-XZ & -YZ & Z
\ea
\right] = \lim_{c \to 0} \left[
\ba{ccc}
0 & 0 & X \\
a XY & b Y^2 & Y \\
XZ & YZ & cZ
\ea
\right] \\
& = \phantom{\lim_{c \to 0}} \left[
\ba{ccc}
0 & 0 & X \\
a XY & b Y^2 & Y \\
XZ & YZ & 0
\ea
\right] = \d(X,Y,XZ, YZ,X^2Y^2Z)(a,b).
\end{align*}
Assume that $a, b \in \CC$ are distinct and non-zero. We have
\begin{align*}
\lim_{c \to \infty} \g(ca, cb) & = \lim_{c \to \infty} \left[
\ba{ccc}
X^2 & XY & X \\
caXY & cbY^2 & Y \\
0 & 0 & Z
\ea
\right] = \lim_{c \to \infty} \left[
\ba{ccc}
a^{-1} X^2 & b^{-1} XY & X \\
XY & Y^2 & c^{-1} Y \\
0 & 0 & Z
\ea
\right] \\
& = 
\left[
\ba{ccc}
a^{-1} X^2 & b^{-1} XY & X \\
XY & Y^2 & 0 \\
0 & 0 & Z
\ea
\right] = \d(X,Z,XY,Y^2,X^2Y^2Z)(a^{-1}, b^{-1}).
\end{align*}
Thus, $S$ contains three affine lines of fixed points for the action of $T$ on $\M_{01}$
denoted $\d_1$, $\d_2$, $\d_3$.
Denote
\[
S_0 = \{ \g(a,b) \mid a,b \in \PP^1, \ a \neq b \} \cup \d_1 \cup \d_2 \cup \d_3.
\]
Denote by $\Delta$ the diagonal of $\PP^1 \times \PP^1$.
From the above calculations it is clear that $S_0$ is isomorphic to an open subset of the
blow-up $B$ of $\PP^1 \times \PP^1$ at $\{ (0,0), (1,1), (\infty, \infty) \}$.
In fact, $S_0 \isom B \setminus \widetilde{\Delta}$, where $\widetilde{\Delta}$ is the strict transform
of $\Delta$. In Section \ref{section_3} we will show that $S \setminus S_0$ consists of a single point (in $\M_1$).
It follows that $S$ is isomorphic to the blow-down of $B$ along $\widetilde{\Delta}$.

From what was said above, we have a complete picture of the fixed locus for the action of $T$ on $\M_0$.
There are thirty isolated points of the form $\a(q_1,q_2)$, twenty-seven projective lines that are the closure
of the affine lines $\b(q,l)$, three surfaces isomorphic to $B \setminus \widetilde{\Delta}$ and a number of
isolated points and affine lines of the form $\d$.
The information about the latter is summarised in the Table 1 below.
We assume that $l_1 = X$, $l_2 = Y$, the other cases being obtained by a permutations of variables.
The first column contains the pair $(q_1,q_2)$ (again modulo permutations of variables),
the second column contains the monomials $d$ of degree $5$
that are in the ideal $(l_1 q_1, l_1 q_2, l_2 q_1, l_2 q_2)$,
the third column contains those $d$ for which $\d(X,Y,q_1,q_2,d)$ is a line,
and in the fourth column are listed those $d$ for which $\d(X,Y,q_1,q_2,d)$ is contained in the closure of a line
or surface of fixed points in $\M_0 \setminus \M_{01}$.
We will use the abbreviation $\Sigma^5 = \{ X^iY^jZ^k, \ i ,j, k \ge 0,\ i+j+k =5 \}$.

\begin{table}[!hpt]{Table 1. Fixed points $\d(X,Y,q_1,q_2,d)$ in $\M_{01}$.}
\begin{center}
\begin{tabular}{|c|l|c|c|}
\hline
$(q_1, q_2)$ & equation of support & affine lines & limit sheaves \\
\hline
$(X^2,Y^2)$
&
$\Sigma^5 \setminus \{ Z^5, Z^4 X, Z^4 Y, Z^3 X^2, Z^3 XY, Z^3 Y^2 \}$
&
$X^2 Y^2 Z$
&
\\
\hline
$(X^2, Z^2)$
&
$\Sigma^5 \setminus \{ Y^5, Z^5, XY^4, Y^4Z, XY^3Z \}$
&
&
$X^3YZ$
\\
\hline
$(Z^2, XY)$
&
$\Sigma^5 \setminus \{ X^5, Y^5, Z^5, X^4Z, Y^4Z \}$
&
&
$X^2Y^2Z$
\\
\hline
$(X^2, YZ)$
&
$\Sigma^5 \setminus \{ Y^5, Z^5, XY^4, XZ^4, X^2 Z^3, Y Z^4 \}$
&
$X^2 YZ^2$
&
$X^3 Y^2$
\\
\hline
$(X^2, XY)$
&
\begin{tabular}{l}
$\Sigma^5 \setminus \{ Y^5, Z^5, XZ^4, X^2 Z^3, Y Z^4, Y^2 Z^3,$ \\
$\phantom{\Sigma^5 \setminus \{} Y^3 Z^2, Y^4 Z, XYZ^3 \}$
\end{tabular}
&
\begin{tabular}{l}
$X^2 YZ^2$ \\ $X^3 YZ$ \\ $X^2 Y^2 Z$ \\ $X^2 Y^3$
\end{tabular}
&
\\
\hline
$(XZ, YZ)$
&
\begin{tabular}{l}
$\Sigma^5 \setminus \{ X^5, Y^5, Z^5, XZ^4, YZ^4, XY^4, X^2 Y^3,$ \\
$\phantom{\Sigma^5 \setminus \{} X^3 Y^2, X^4 Y \}$
\end{tabular}
&
\begin{tabular}{l}
$XYZ^3$ \\ $XY^3 Z$ \\ $X^2 Y^2 Z$ \\ $X^3 YZ$
\end{tabular}
&
$X^2 Y^2 Z$
\\
\hline
$(X^2, XZ)$
&
\begin{tabular}{l}
$\Sigma^5 \setminus \{ Y^5, Z^5, XY^4, XZ^4, YZ^4, Y^2 Z^3,$ \\
$\phantom{\Sigma^5 \setminus \{} Y^3 Z^2, Y^4 Z \}$
\end{tabular}
&
\begin{tabular}{l}
$X^3 Z^2$ \\ $X^2 YZ^2$ \\ $X^2 Y^2 Z$
\end{tabular}
&
$X^4 Y$
\\
\hline
$(XY, YZ)$
&
\begin{tabular}{l}
$\Sigma^5 \setminus \{ X^5, Y^5, Z^5, YZ^4, XZ^4, X^2 Z^3,$ \\
$\phantom{\Sigma^5 \setminus \{} X^3 Z^2, X^4 Z \}$
\end{tabular}
&
\begin{tabular}{l}
$XY^2 Z^2$ \\ $X^2 YZ^2$ \\ $X^3 YZ$
\end{tabular}
&
$X^2 Y^3$
\\
\hline
$(XZ, Z^2)$
&
\begin{tabular}{l}
$\Sigma^5 \setminus \{ X^5, Y^5, Z^5, Y^4 Z, XY^4, X^2 Y^3,$ \\
$\phantom{\Sigma^5 \setminus \{} X^3 Y^2, X^4 Y \}$
\end{tabular}
&
\begin{tabular}{l}
$XY^2 Z^2$ \\ $X^2 YZ^2$ \\ $X^3 Z^2$
\end{tabular}
&
$X^2 YZ^2$
\\
\hline
\end{tabular}
\end{center}
\end{table}


\subsection{Fixed points in $\M_1$}
\label{2.3}

Assume that the point in $\M_{\PP^2}(5,3)$ represented by the morphism
\[
\f = \left[
\ba{cc}
\f_{11} & 0 \\
\f_{21} & \f_{22}
\ea
\right] = \left[
\ba{cccc}
l_1 & l_2 & 0 & 0 \\
q_{11} & q_{12} & l_{11} & l_{12} \\
q_{21} & q_{22} & l_{21} & l_{22} \\
q_{31} & q_{32} & l_{31} & l_{32}
\ea
\right]
\]
is fixed by $T$. Then, for each $t \in T$, there is $(g(t), h(t))\in G_1$ such that $t\f = h(t) \f g(t)$.
We write
\[
g(t) = \left[
\ba{cc}
g_1 & 0 \\
u & g_2
\ea
\right], \qquad h(t) = \left[
\ba{cc}
h_{11} & 0 \\
v & h_2
\ea
\right].
\]
From the relations $t \f_{11} = h_{11}(t) \f_{11} g_1(t)$ and $t \f_{22} = h_2(t) \f_{22} g_2(t)$
we see that $\f_{11}$ gives a $T$ fixed point in $\N(3,2,1)$ while $\f_{22}$ gives a $T$-fixed point in $\N(3,2,3)$.
The fixed points in $\N(3,2,1)$ have been described at \ref{2.1.3}.
They are completely determined by the $T$-fixed point $x \in \PP^2$ given by the ideal $(l_1,l_2)$.
The fixed points in $\N(3,2,3)$ have been described in \cite[Chapter VI]{drezet_cohomologie}.
They are completely determined by the ideal generated by the maxinal minors $q_1$, $q_2$, $q_3$ of $\f_{22}$.
We have ten fixed points corresponding to the ideals of $T$-fixed subschemes $\Z$ of $\PP^2$
of length $3$, where $\Z$ is not contained in a line,
and three more points corresponding to the ideals $(X^2, XY, XZ)$, $(XY, Y^2, YZ)$, $(XZ, YZ, Z^2)$.
Let $\M(l_1,l_2,q_1,q_2,q_3)$ be the image in $\M_{\PP^2}(5,3)$ of the set of morphisms
$\f$ as above with fixed $\f_{11}$ and $\f_{22}$. Given a monomial $d$ of degree $5$, we denote by
$\M(l_1,l_2,q_1,q_2,q_3,d)$ the subset given by the additional condition that $\det(\f)$ be a multiple of $d$.

\begin{proposition}
\label{2.3.1}
Assume that $\f_{11}$ and $\f_{22}$ give $T$-fixed points in $\N(3,2,1)$, respectively $\N(3,2,3)$.
Let $l_1$, $l_2$ be the entries of $\f_{11}$ and let $q_1$, $q_2$, $q_3$ be the maximal minors of $\f_{22}$.
Then, for any monomial $d$ of degree $5$ belonging to the ideal
$(l_1 q_1, l_1 q_2, l_1 q_3, l_2 q_1, l_2 q_2, l_2 q_3)$, the set of fixed points for the action of $T$ on
$\M(l_1,l_2,q_1,q_2,q_3,d)$ has precisely one irreducible component, which is either a point or an affine line.
\end{proposition}

\begin{proof}
We will only examine the case when $\Z$ is a triple point supported on $\{ x \}$, the other cases being analogous.
Consider morphisms of the form
\[
\f = \left[
\ba{cccc}
X & Y & 0 & 0 \\
q_{11} & q_{12} & X & 0 \\
q_{21} & q_{22} & 0 & Y \\
q_{31} & q_{32} & Y & X
\ea
\right].
\]
Performing elementary operations on $\f$ we may assume that
\[
q_{11} \in \CC[Y,Z], \quad q_{12} \in \CC[Z], \quad q_{21} \in \CC[Z],
\quad q_{22} \in \CC[X,Z], \quad q_{32} \in \CC[X,Z].
\]
We have the relations
\[
t \f_{11} = \f_{11} \left[
\ba{cc}
t_0 & 0 \\
0 & t_1
\ea
\right], \qquad t \f_{22} = \left[
\ba{ccc}
t_0^{} t_1^{-1} & 0 & 0 \\
0 & t_0^{-1} t_1^{} & 0 \\
0 & 0 & 1
\ea
\right] \f_{22} \left[
\ba{cc}
t_1 & 0 \\
0 & t_0
\ea
\right].
\]
The morphisms $\f_{11}$ and $\f_{22}$ are stable as Kronecker modules, hence they have trivial
stabilisers in $\GL(2,\CC)$, respectively in $(\GL(2, \CC) \times \GL(3,\CC))/\CC^*$.
Thus, we may assume that
\[
g_1 (t) = h_{11}^{-1} \left[
\ba{cc}
t_0 & 0 \\
0 & t_1
\ea
\right], \qquad g_{2} = \left[
\ba{cc}
t_1 & 0 \\
0 & t_0
\ea
\right], \qquad h_2 = \left[
\ba{ccc}
t_0^{} t_1^{-1} & 0 & 0 \\
0 & t_0^{-1} t_1^{} & 0 \\
0 & 0 & 1
\ea
\right].
\]
The relation $t \f = h(t) \f g(t)$ is equivalent to the relation
$t \f_{21} = v \f_{11} g_1 + h_2 \f_{21} g_1 + h_2 \f_{22} u$, or
\begin{align*}
\left[
\ba{cc}
t q_{11} & t q_{12} \\
t q_{21} & t q_{22} \\
t q_{31} & t q_{32}
\ea
\right] = & \left[
\ba{c}
v_1 \\ v_2 \\ v_3
\ea
\right] \left[
\ba{cc}
X & Y
\ea
\right] h_{11}^{-1} \left[
\ba{cc}
t_0 & 0 \\
0 & t_1
\ea
\right] +
\\
& \left[
\ba{ccc}
t_0^{} t_1^{-1} & 0 & 0 \\
0 & t_0^{-1} t_1^{} & 0 \\
0 & 0 & 1
\ea
\right] \left[
\ba{cc}
q_{11} & q_{12} \\
q_{21} & q_{22} \\
q_{31} & q_{32}
\ea
\right] h_{11}^{-1} \left[
\ba{cc}
t_0 & 0 \\
0 & t_1
\ea
\right] +
\\
& \left[
\ba{ccc}
t_0^{} t_1^{-1} & 0 & 0 \\
0 & t_0^{-1} t_1^{} & 0 \\
0 & 0 & 1
\ea
\right] \left[
\ba{cc}
X & 0 \\
0 & Y \\
Y & X
\ea
\right] \left[
\ba{cc}
u_{11} & u_{12} \\
u_{21} & u_{22}
\ea
\right].
\end{align*}
From the relation
\[
t q_{12} = h_{11}^{-1} t_1^{} v_1^{} Y + h_{11}^{-1} t_0^{} q_{12}^{} + t_0^{} t_1^{-1} u_{12}^{} X
\]
we get the relations
\[
\tag{1}
t q_{12}^{} = h_{11}^{-1} t_0^{} q_{12}^{}, \quad u_{12}^{} = aY, \quad v_1^{} = - h_{11}^{} t_0^{} t_1^{-2} aX
\quad \text{for some $a \in \CC$}.
\]
From the relation
\[
t q_{11}^{} = h_{11}^{-1} t_0^{} v_1^{} X + h_{11}^{-1} t_0^2 t_1^{-1} q_{11}^{} + t_0^{} t_1^{-1} u_{11}^{} X
\]
we get the relations
\[
\tag{2}
t q_{11}^{} = h_{11}^{-1} t_0^2 t_1^{-1} q_{11}^{}, \qquad u_{11}^{} = t_0^{} t_1^{-1} aX.
\]
From the relation
\[
t q_{21}^{} = h_{11}^{-1} t_0^{} v_2^{} X + h_{11}^{-1} t_1^{} q_{21}^{} + t_0^{-1} t_1^{} u_{21}^{} Y
\]
we get the relations
\[
\tag{3}
t q_{21}^{} = h_{11}^{-1} t_1^{} q_{21}^{}, \quad u_{21}^{} = bX, \quad v_2^{} = -h_{11}^{} t_0^{-2} t_1^{} bY, \quad
\text{for some $b \in \CC$}.
\]
From the relation
\[
t q_{22}^{} = h_{11}^{-1} t_1^{} v_2^{} Y + h_{11}^{-1} t_0^{-1} t_1^{2} q_{22}^{} + t_0^{-1} t_1^{} u_{22}^{} Y
\]
we get the relations
\[
\tag{4}
t q_{22}^{} = h_{11}^{-1} t_0^{-1} t_1^{2} q_{22}^{}, \qquad u_{22}^{} = t_0^{-1} t_1^{} bY.
\]
From the relation
\[
t q_{32}^{} = h_{11}^{-1} t_1^{} v_3^{} Y + h_{11}^{-1} t_1^{} q_{32}^{} + u_{12}^{} Y + u_{22}^{} X
\]
we get the relation
\[
\tag{5}
t q_{32}^{} = h_{11}^{-1} t_1^{} q_{32}^{}, \qquad v_3^{} = -h_{11}^{}(t_1^{-1} aY + t_0^{-1} bX).
\]
Finally, we have the relations
\begin{align*}
t q_{31}^{} & = h_{11}^{-1} t_0^{} v_3^{} X + h_{11}^{-1} t_0^{} q_{31}^{} + u_{11}^{} Y + u_{21}^{} X \\
& = h_{11}^{-1} t_0^{} q_{31}^{} - t_0^{} (t_1^{-1} aY + t_0^{-1} bX)X + t_0^{} t_1^{-1} aXY + bX^2, \\
t q_{31}^{} & = h_{11}^{-1} t_0^{} q_{31}^{}, \tag{6}
\end{align*}
Combining relations (1)--(6), yields the equation $t \f_{21} = h_2(t) \f_{21} g_1(t)$.
It becomes clear now that $q_{rs}$ are monomials and $h_{11}^{}(t) = t_0^{-i} t_1^{-j} t_2^{-k}$,
where $i$, $j$, $k$ are integers satisfying $i + j+ k = 1$.
In fact,
\[
\f_{21} = \left[
\ba{ll}
c_{11} X^{2+i} Y^{-1+j} Z^k & c_{12} X^{1+i} Y^j Z^k \\
c_{21} X^i Y^{1+j} Z^k & c_{22} X^{-1+i} Y^{2+j} Z^k \\
c_{31} X^{1+i} Y^j Z^k & c_{32} X^i Y^{1+j} Z^k
\ea
\right].
\]
Assume that $c_{12} \neq 0$. Then $i = -1$, $j=0$, $c_{11}=0$, $c_{21}=0$, $c_{22}=0$, $c_{32}=0$.
We get the $T$-fixed points represented by the matrices
\[
\f_1(c) = \left[
\ba{cccc}
X & Y & 0 & 0 \\
0 & Z^2 & X & 0 \\
0 & 0 & 0 & Y \\
cZ^2 & 0 & Y & X
\ea
\right], \qquad c \in \CC \setminus \{ -1 \}.
\]
Assume that $c_{12}=0$, $c_{21} \neq 0$. Then $i=0$, $j=-1$, $c_{11}=0$, $c_{22}=0$, $c_{31}=0$.
We obtain the fixed points
\[
\f_2(c) = \left[
\ba{cccc}
X & Y & 0 & 0 \\
0 & 0 & X & 0 \\
Z^2 & 0 & 0 & Y \\
0 & cZ^2 & Y & X
\ea
\right], \qquad c \in \CC \setminus \{ -1 \}.
\]
Assume that $c_{12}=0$, $c_{21}=0$, $c_{11} \neq 0$.
Then $i=-2$, $c_{22}=0$, $c_{31}=0$, $c_{32}=0$, and we get the fixed points
\[
\left[
\ba{cccc}
X & Y & 0 & 0 \\
q & 0 & X & 0 \\
0 & 0 & 0 & Y \\
0 & 0 & Y & X
\ea
\right], \qquad q \in \{ Y^2, YZ, Z^2 \}.
\]
Assume that $c_{11} = 0$, $c_{12}=0$, $c_{21}=0$, $c_{22} \neq 0$.
Then $j =-2$, $c_{31}=0$, $c_{32}=0$, and we get the fixed points
\[
\left[
\ba{cccc}
X & Y & 0 & 0 \\
0 & 0 & X & 0 \\
0 & q & 0 & Y \\
0 & 0 & Y & X
\ea
\right], \qquad q \in \{ X^2, XZ, Z^2 \}.
\]
Assume that $c_{11}=0$, $c_{12}=0$, $c_{21}=0$, $c_{22}=0$, $c_{32} \neq 0$.
Then $j = -1$, $c_{31}=0$, and we obtain the fixed points
\[
\left[
\ba{cccc}
X & Y & 0 & 0 \\
0 & 0 & X & 0 \\
0 & 0 & 0 & Y \\
0 & q & Y & X
\ea
\right], \qquad q \in \{ X^2, XZ, Z^2\}.
\]
Note that for $q = Z^2$ we have the point $\f_2(\infty)$.
Assume, finally, that $c_{11}=0$, $c_{12}=0$, $c_{21}=0$, $c_{22}=0$, $c_{32}=0$, $c_{31} \neq 0$.
We have the fixed points
\[
\left[
\ba{cccc}
X & Y & 0 & 0 \\
0 & 0 & X & 0 \\
0 & 0 & 0 & Y \\
q & 0 & Y & X
\ea
\right], \qquad q \in \{ X^2, Y^2, Z^2, XY, XZ, YZ \}.
\]
Note that for $q=Z^2$ we have the point $\f_1 (\infty)$.
In conclusion, the set of fixed points for the action of $T$ on $\M(X,Y,X^2,XY,Y^2)$
consists of thirteen isolated points and two affine lines, namely
\[
\{ [\f_1(c)], \ c \in \PP^1 \setminus \{-1 \} \}, \qquad \{ [\f_2(c)], \ c \in \PP^1 \setminus \{-1 \} \}.
\qedhere
\]
\end{proof}


\subsection{Fixed points in $\M_2$}
\label{2.4}

Assume that $\f \in W_2$ gives a $T$-fixed point $\F$ in $\M_{\PP^2}(5,3)$.
Then it is easy to see that $\f_{11}$ gives a fixed point in the Kronecker moduli space $\N(3,3,2)$.
Given quadratic forms $q_1$, $q_2$, $q_3$, we denote by
$\M(q_1, q_2, q_3)$ be the image in $\M_{\PP^2}(5,3)$
of the set of morphisms $\f \in W_2$ such that $(q_1, q_2, q_3)$ is the ideal generated
by the maximal minors of $\f_{11}$.
Given a monomial $d$ of degree $5$, we denote by
$\M(q_1,q_2,q_3,d)$ the subset given by the additional condition that $\det(\f)$ be a multiple of $d$.

As mentioned at the beginning of this section, the points in $\M_{20}$ are of the form $\O_Q(\Z)(1)$.
Clearly, $\O_Q(\Z)(1)$ is fixed by $T$ if and only if $Q$ and $\Z$ are $T$-invariant,
so we will concentrate on finding the fixed points $\F$ in $\M_{21}$.
As mentioned at the beginning of this section, $\F$ is an extension of $\O_L$ by $\O_C(1)$.
Denote
\[
\M_{L,C} = \PP(\Ext^1(\O_L, \O_C(1)) \cap \M_{21}.
\]
The common factor of the maximal minors of $\f_{11}$ is a monomial
because $\f_{11}$ gives a fixed point in the Kronecker moduli space.
Thus $L$ is $T$-invariant, and the same is true of $C$.
We will assume that $L$ is given by the equation $X=0$, the other cases being analogous.
Write $d = \det(\f) = X^i Y^j Z^k$ and fix $f_1, f_2, f_3 \in \Sym^3 V^*$ satisfying the equation
\[
-f_1 Y - f_2 Z + f_3 X = d/l = X^{i-1} Y^j Z^k.
\]
Consider the set $U$ of morphisms of the form
\[
\left[
\ba{ccc}
X & 0 & Y \\
0 & X & Z \\
f_1- qZ & f_2 + qY & f_3
\ea
\right], \qquad q = aY^2 + bYZ + cZ^2, \quad a,b,c \in \CC.
\]
By the argument at \cite[Proposition 5.1]{choi_maican}, we can show that the map
$U \to \M_{L,C}$ sending $\f$ to $[\Coker(\f)]$ is an isomorphism and the induced action
of $T$ on $U$ via this isomorphism is given by $(t,q) \mapsto t_0^{1-i} t_1^{1-j} t_2^{1-k} (tq)$.
Choosing coordinates $(a,b,c)$ we identify $U$ with $\AA^3$.
The induced action of $T$ on $\AA^3$ is given by
\[
t(a,b,c) = t_0^{1-i}(t_1^{3-j} t_2^{1-k}a, \ t_1^{2-j} t_2^{2-k} b, \ t_1^{1-j} t_2^{3-k} c).
\]
If $i \neq 1$, or if $(i,j,k) \in \{ (1,4,0), \ (1,0,4) \}$, we get a single isolated fixed point, namely
$(0,0,0)$. If $(i,j,k) \in \{ (1,3,1),\ (1,2,2), \ (1,1,3) \}$, then we get an affine line of fixed points.
Summarising we obtain the following proposition.

\begin{proposition}
\label{2.4.1}
Assume that $\f_{11}$ gives a $T$-fixed point in $\N(3,3,2)$.
Let $q_1$, $q_2$, $q_3$ be the maximal minors of $\f_{11}$.
Then, for any monomial $d$ of degree $5$ belonging to the ideal $(q_1, q_2, q_3)$,
the set of fixed points for the action of $T$ on $\M(q_1,q_2,q_3,d)$ has precisely one irreducible
component, which is either a point or an affine line.
We have a line in the following cases:

\begin{center}
\begin{tabular}{lr}
$(q_1, q_2, q_3)$ & $d$ \\
\hline
$(X^2, XY, XZ)$ & $XY^3Z$, $XY^2 Z^2$, $XYZ^3$ \\
$(XY, Y^2, YZ)$ & $X^3 YZ$, $X^2 YZ^2$, $XYZ^3$ \\
$(XZ, YZ, Z^2)$ & $X^3 YZ$, $X^2 Y^2 Z$, $XY^3 Z$
\end{tabular}
\end{center}
\end{proposition}


\subsection{Fixed points in $\M_3$}
\label{2.5}

Clearly, $\O_Q(-\Y)(2)$ gives a $T$-fixed point in $\M_{\PP^2}(5,3)$ if and only if $Q$ and $\Y$ are $T$-invariant.
Equivalently, the fixed points in $\M_3$ are given by matrices $\f \in W_3$ that have monomial entries.


\section{The torus representation of the tangent spaces at the fixed points}
\label{section_3}

Let $\F = \Coker(\f)$ give a $T$-fixed point in $\M_k$, where $\f \in W_k$.
Since $W_k \to \M_k$ is a geometric quotient map, the tangent space
$\T_{[\F]} \M_k$ is the quotient of $\T_{\f} W_k$ by the tangent space
$\T_{\f} (G_k \f)$ to the orbit of $\f$.
The latter is a $T$-invariant subspace because $[\F]$ is $T$-fixed.
Thus, the list of weights for the action of $T$ on $\T_{[\F]} \M_k$
is obtained by subtracting the list of weights for $\T_{\f} (G_k \f)$ from the list of
weights for $\T_{\f} W_k$.
These two lists can be determined as in \cite[Section 6]{choi_maican}
provided that there is a morphism of groups
\[
(\CC^*)^3  \to G_k, \quad t \mapsto ((u(t), v(t)),
\quad \text{such that} \quad t \f = v(t) \f u(t) \quad \text{for all} \quad t \in (\CC^*)^3
\]
and such that $u(t)$ and $v(t)$ are diagonal matrices.
The diagonal entries $u_i$, $v_j$ of $u$, $v$ are characters of $(\CC^*)^3$.
The existence of $u$ and $v$ is obvious when $\f$ is $\a$, $\b$ or $\g$ from \ref{2.1}.
In the sequel, for all other fixed points $[\F]$ in $\M_{\PP^2}(5,3)$ we will give $\f$
for which $u$ and $v$ exist.

For convenience, we will use additive notation when we deal with characters of $(\CC^*)^3$ or of $T$.
Denote by $x$, $y$, $z$ the standard basis for the lattice of characters of $(\CC^*)^3$.
The characters of $T$ are of the form $ix+jy+kz$, $i, j, k \in \ZZ$, $i+j+k=0$.
We will use the following abbreviations:
\begin{align*}
\Sigma^5 & = \{ X^iY^jZ^k, \ i ,j, k \ge 0,\ i+j+k =5 \}, \\
\chi_0 & = \text{the trivial character of $(\CC^*)^3$ or of $T$}, \\
\s^0 & = \{ x - y, \ y - x,\ x - z,\ z - x,\ y - z,\ z - y \}, \\
\s^l & = \{ ix + jy + kz, \ i,j,k \in \ZZ, \ i,j,k \ge 0,\ i + j + k = l \}, \quad l \ge 1,\\
\s^l_x & = \s^l \setminus (x + \s^{l - 1}), \\
\s^l_y & = \s^l \setminus (y + \s^{l - 1}), \\
\s^l_z & = \s^l \setminus (z + \s^{l - 1}).
\end{align*}
We also adopt the following convention: whenever a monomial $f = X^i Y^j Z^k$ appears in a list of characters,
it stands for the expression $ix + jy + kz$.


\subsection{Fixed points in $\M_0$}
\label{3.1}

According to \cite[6.1.1]{choi_maican}, the action of $T$ on $\T_{\f} W_0$ is given by the formula
\[
\tag{3.1.1}
(t,w) \mapsto v(t)^{-1} (tw) u(t)^{-1},
\]
where $tw$ refers to the canonical action of $(\CC^*)^3$ on the symmetric powers of $V^*$.
Thus, the list of weights for the action of $T$ on $\T_{\f} W_0$ is represented by the array
\[
\ba{rrrrr}
-v_1 - u_1 + \s^2 & \quad & -v_1 - u_2 + \s^2 & \quad & - v_1 - u_3 + \s^1 \\
-v_2 - u_1 + \s^2 & \quad & -v_2 - u_2 + \s^2 & \quad & - v_2 - u_3 + \s^1 \\
-v_3 - u_1 + \s^2 & \quad & -v_3 - u_2 + \s^2 & \quad & - v_3 - u_3 + \s^1
\ea.
\]
According to \cite[6.1.3]{choi_maican}, the action of $T$ on $\T_{\f}(G_0 \f)$,
which is identified with the tangent space of $G_0$ at the neutral element,
is given by the formula
\[
\tag{3.1.2}
t(A,B) = (u(t) (tA) u(t)^{-1}, \ v(t)^{-1} (tB) v(t)).
\]
It follows that the list of weights for the action of $T$ on $\T_{\f}(G_0 \f)$ is represented by the array
\[
\ba{lllll}
\chi_0 & u_1 - u_2 & \phantom{-} \chi_0 & - v_1 + v_2 & - v_1 + v_3 \\
u_2 - u_1 & \chi_0 & - v_2 + v_1 & \phantom{-} \chi_0 & - v_2 + v_3 \\
u_3 - u_1 + \s^1 & u_3 - u_2 + \s^1 & - v_3 + v_1 & - v_3 + v_2 & \phantom{-} \chi_0
\ea.
\]
Assume now that
\[
\f = \a(q_1, q_2) = \left[
\ba{ccc}
q_1 & 0 & X \\
0 & q_2 & Y \\
0 & 0 & Z
\ea
\right].
\]
We have
\[
\ba{lll}
u_1 = q_1 - x & \qquad & v_1 = x \\
u_2 = q_2 - y & \qquad & v_2 = y \\
u_3 = 0 & \qquad & v_3 = z
\ea.
\]
Thus, the list of weights for the action of $T$ on $\T_{[\a]} \M$ is represented by the following table:
\[
\ba{rrr}
- q_1 + (\s^2 \setminus \{ q_1 \}) & \quad & - q_2 + (\s^2 \setminus \{ q_2 \}) \\
x - y - q_1 + (\s^2 \setminus \{ q_2 \}) & \quad & y - x - q_2 + (\s^2 \setminus \{ q_1 \}) \\
x - z - q_1 + \s^2_z & \quad & y - z - q_2 + \s^2_z
\ea.
\]
Here $(q_1, q_2) \in (\s^2_x \times \s^2_y) \cup \{ (2x, x+y), \ (x+y, 2y), \ (x+z, y+z) \}$.
Assume next that
\[
\f = \b(q,l) = \left[
\ba{ccc}
q & 0 & X \\
0 & alY & Y \\
0 & blZ & Z
\ea
\right].
\]
We have
\[
\ba{lll}
u_1 = q - x & \qquad & v_1 = x \\
u_2 = l & \qquad & v_2 = y \\
u_3 = 0 & \qquad & v_3 = z
\ea.
\]
Thus, the list of weights for the action of $T$ on $\T_{[\b]} \M$ is represented by the following array:
\[
\ba{rrr}
- q + (\s^2 \setminus \{ q \}) & \quad & -z - l + (\s^2 \setminus \{ q \}) \\
- y + z - q + (\s^2 \setminus \{ y + l \}) & \quad & -x - l + (\s^2 \setminus \{ x + l \}) \\
- x + z - q + \s^2_x & \quad & -y - l + \s^2_y
\ea.
\]
Here $q \in \s^2_x$, $l \in \s^1$. We next examine the case of the fixed surfaces.
Assume that
\[
\f = \g = \left[
\ba{ccc}
X^2 & XY & X \\
aXY & bY^2 & Y \\
0 & 0 & Z
\ea
\right].
\]
Clearly, we have
\[
\ba{lll}
u_1 = x & \qquad & v_1 = x \\
u_2 = y & \qquad & v_2 = y \\
u_3 = 0 & \qquad & v_3 = z
\ea.
\]
The list of weights for the action of $T$ on $\T_{[\g]} \M$ is represented by the following array:
\[
\ba{rrrrr}
- 2x + (\s^2 \setminus \{ 2x \}) & \quad & - x - y + (\s^2 \setminus \{ 2x \}) & \quad & - x - z + \s^2_z \\
- 2y + (\s^2 \setminus \{ 2y \}) & \quad & - x - y + (\s^2 \setminus \{ 2y \}) & \quad & - y - z + \s^2_z
\ea.
\]
Given $q_1$, $q_2$ and $d$ as in Table 1, Section \ref{2.2}, we consider morphisms of the form
\[
\f = \d(X, Y, q_1, q_2, d) = \left[
\ba{ccc}
c_{11} d/q_2 Y & c_{12} d/q_1 Y & X \\
c_{21} d/q_2 X & c_{22} d/q_1 X & Y \\
q_1 & q_2 & 0
\ea
\right].
\]
We have
\[
\ba{lll}
u_1 = d - x - y - q_2 & \qquad & v_1 = x \\
u_2 = d - x - y - q_1 & \qquad & v_2 = y \\
u_3 = 0 & \qquad & v_3 = - d + q_1 + q_2 + x + y
\ea.
\]
The list of weights for the action of $T$ on $\T_{[\d]} \M$ is obtained by removing an extra copy of $\chi_0$
from the following table:
\[
\ba{rrr}
- d + q_2 + y + \s^2_x & \quad & - d + q_2 + x + (\s^2 \setminus \{ q_1 \}) \\
- d + q_1 + x + \s^2_y & \quad & - d + q_1 + y + (\s^2 \setminus \{ q_2 \}) \\
d - q_1 - q_2 - x - y + z & \quad & - q_1 + (\s^2 \setminus \{ q_1, q_2 \}) \\
- x + z & \quad & - q_2 + (\s^2 \setminus \{ q_1, q_2 \}) \\
- y + z
\ea.
\]
It is known that $\dim(\T_{[\F]} \M)_{\chi_0}$
equals the dimension of the irreducible component of $\M^{T}$ that contains $[\F]$.
Thus, counting how many times $\chi_0$ appears in the above list,
gives us another approach for determining column four of Table 1 in Section 2.2.
For instance, $\dim(\T_{[\f]})_{\chi_0} = 2$ for $\f = \d(X, Y, XZ, YZ, X^2 Y^2 Z)$ or $\d(X, Y, XZ, Z^2, X^2 Y Z^2)$
even though these points belong to irreducible components of $\M_{01}^{T}$ of dimension $1$.
This shows that these points belong to the closure of irreducible components of
\[
(\M \setminus \overline{\M}_{01})^{T} = (\M_0 \setminus \M_{01})^{T}
\]
of dimension $2$.
For the other points in column four of Table 1 $\dim(\T_{[\f]})_{\chi_0} = 1$ even though they belong to
isolated points of $\M_{01}^{T}$; thus they belong to the closure of lines in $(\M_0 \setminus \M_{01})^{T}$.

All fixed points in $\M_0$ can be obtained from the points $\a$, $\b$, $\g$, $\d$ above by a permutation of variables.
We have thus found the $T$-representation for all tangent spaces at fixed points in $\M_0$.


\subsection{Fixed points in $\M_1$}
\label{3.2}

The list of weights for the action of $T$ on $\T_{\f} W_1$ given by formula (3.1.1) is represented by the array
\[
\ba{rrrrrrr}
- v_1 - u_1 + \s^1 & \quad & - v_1 - u_2 + \s^1 & & & & \\
- v_2 - u_1 + \s^2 & \quad & - v_2 - u_2 + \s^2 & \quad & - v_2 - u_3 + \s^1 & \quad & - v_2 - u_4 + \s^1 \\
- v_3 - u_1 + \s^2 & \quad & - v_3 - u_2 + \s^2 & \quad & - v_3 - u_3 + \s^1 & \quad & - v_3 - u_4 + \s^1 \\
- v_4 - u_1 + \s^2 & \quad & - v_4 - u_2 + \s^2 & \quad & - v_4 - u_3 + \s^1 & \quad & - v_4 - u_4 + \s^1
\ea.
\]
The list of weights for the action of $T$ on $\T_{e} G_1$ given by formula (3.1.2) is represented by the table
\[
\ba{rllllll}
- v_2 + v_1 + \s^1 & \quad & \phantom{-} \chi_0 & \quad & - v_2 + v_3 & \quad & - v_2 + v_4 \\
- v_3 + v_1 + \s^1 & \quad & - v_3 + v_2 & \quad & \phantom{-} \chi_0 & \quad & - v_3 + v_4 \\
- v_4 + v_1 + \s^1 & \quad & - v_4 + v_2 & \quad & - v_4 + v_3 & \quad & \phantom{-} \chi_0
\ea
\]
\[
\ba{lllllll}
\chi_0 & \quad &  u_1 - u_2 & & & & \\
u_2 - u_1 & \quad & \chi_0 & & & & \\
u_3 - u_1 + \s^1 & \quad & u_3 - u_2 + \s^1 & \quad & \chi_0 & \quad & u_3 - u_4 \\
u_4 - u_1 + \s^1 & \quad & u_4 - u_2 + \s^1 & \quad & u_4 - u_3 & \quad & \chi_0
\ea.
\]
There is a small complication here, namely every $\f \in W_1$ has a stabiliser of dimension $2$
consisting of matrices of the form
\[
\left(
\left[
\ba{cc}
I_2 & 0 \\
C \f_{11} & I_2
\ea
\right], \ \left[
\ba{cc}
1 & 0 \\
\f_{22} C & I_3
\ea
\right] \right), \quad \text{where} \quad C = \left[
\ba{c}
c_1 \\
c_2
\ea
\right], \quad c_1, c_2 \in \CC.
\]
We have, therefore, an isomorphism $\CC^2 \isom \T_e \Stab(\f)$ given by
\[
(c_1, c_2) \longmapsto \left( \left[
\ba{cc}
0 & 0 \\
C \f_{11} & 0
\ea
\right], \ \left[
\ba{cc}
0 & 0 \\
\f_{22} C & 0
\ea
\right] \right).
\]
We denote by $s_1$, $s_2$, the images of $(1,0)$, respectively, $(0,1)$.
In order to get the list of weights for $\T_{\f} (G_1 \f)$, we need to subtract the list of weights for $\T_e \Stab(\f)$,
which we will determine case by case,
from the list of weights for $\T_{e} G_1$ above.

\begin{proposition}
\label{3.2.1}
Let $N_{[\F]}$ be the normal space to $\M_1$ at $[\F]$.
The torus $T$ acts on $N_{[\F]}$ with weights $- v_1 - u_3$, $- v_1 - u_4$.
\end{proposition}

\begin{proof}
By an argument analogous to the argument at \cite[Theorem 4.3.3]{dedicata},
we can show that $\M_0 \cup \M_{01}$ is the good quotient of an open subset $\WW \subset \WW_1$ modulo $G_1$.
Here $\WW$ is simply the set of injective morphisms that have semi-stable cokernel.
The list of weights for the action of $T$ on $\T_{\f} \WW_1$ given by formula (3.1.1) is the same as the list for $\T_{\f} W_1$,
except that it contains $- v_1 - u_3$ and $- v_1 - u_4$ in the upper-right corner,
accounting for the two normal directions to $W_1$ at $\f$.
\end{proof}

\noindent
Recall from Proposition \ref{2.3.1} that an irreducible component of $\M_1^{T}$ is uniquely determined by
$l_1$, $l_2$, $q_1$, $q_2$, $q_3$ and $d$.
The ideal $(q_1, q_2, q_3)$ defines a $T$-invariant zero-dimensional scheme $\Z$ of length $3$ that is not contained
in a line or is of the form $(lX, lY, lZ)$, $l \in \{ X, Y, Z \}$.
Assume firstly that $\f$ has the form
\[
{
\mathversion{bold}
\boldsymbol{\e(d)} = \left[
\ba{cccc}
X & Y & 0 & 0 \\
c_{11} d/Y^2 Z & c_{12} d/XYZ & X & X \\
c_{21} d/XYZ & c_{22} d/X^2 Z & Y & 0 \\
c_{31} d/XY^2 & c_{32} d/X^2 Y & 0 & Z
\ea
\right]
}
\]
corresponding to the case when $\Z$ consists of three distinct points,
namely $(1, 0, 0)$, $(0, 1, 0)$, $(0, 0, 1)$.
We have
\[
\ba{lll}
u_1 = d - y & \qquad & v_1 = x + y - d \\
u_2 = d - x & \qquad & v_2 = - y - z \\
u_3 = x + y + z & \qquad & v_3 = - x - z \\
u_4 = x + y + z & \qquad & v_4 = - x - y
\ea.
\]
The torus $T$ acts on both $\CC s_1$ and $\CC s_2$ with weight $- d + 2x + 2y+ z$.
The list of weights for the action of $T$ on $\T_{[\e (d)]} \M$ is obtained by removing an extra copy of $\chi_0$
from the following table:
\[
\ba{lll}
\phantom{-} d - 2x - 2y - z & \quad & - d + 2y + z + \s^2_x \\
\phantom{-} d - 2x - 2y - z & \quad & - d + x + y + z + \s^2_x \\
- d + 2x + 2y + z & \quad & - d + x + y + z + \s^2_y \\
- d + 4x + y & \quad & - d + x + 2y + \s^2_z \\
\phantom{-} z - x & \quad & - d + 2x + z + \s^2_y \\
\phantom{-} z - y & \quad & \phantom{-} \s^0
\ea.
\]
Consider next morphisms of the form
\[
{
\mathversion{bold}
\boldsymbol{\z (l_1, l_2, d)} = \left[
\ba{cccc}
l_1 & l_2 & 0 & 0 \\
c_{11} d/l_2 X^2 & c_{12} d/l_1 X^2 & Y & 0 \\
c_{21} d/l_2 YZ & c_{22} d/l_1 YZ & 0 & X \\
c_{31} d/l_2 XY & c_{32} d/l_1 XY & X & Z
\ea
\right]
}
\]
corresponding to the case when $\Z$ is the union of a double point and a simple point.
We have
\[
\ba{lll}
u_1 = d - l_2 & \qquad & v_1 = - d + l_1 + l_2 \\
u_2 = d - l_1 & \qquad & v_2 = - 2x \\
u_3 = 2x + y & \qquad & v_3 = - y - z \\
u_4 = x + y + z & \qquad & v_4 = - x - y
\ea.
\]
The torus $T$ acts on $\CC s_1$ with weight $2x + y - d + l_1 + l_2$
and on $\CC s_2$ with weight $x + y + z - d + l_1 + l_2$.
The list of weights for the action of $T$ on $\T_{[\z]} \M$ is obtained by subtracting the list
$\{ \chi_0, \ - d + l_1 + l_2 + x + 2y \}$ from the list
\[
\ba{lllll}
d - l_1 - l_2 - 2x - y & \quad & - d + l_2 + 2x + \s^2_{l_1} & \quad & x - y \\
d - l_1 - l_2 - x - y - z & \quad & - d + l_1 + y + z + \s^2_{l_2} & \quad & y - x \\
l_3 - l_1 & \quad & - d + l_2 + y + z + \s^2_x & \quad & z - x \\
l_3 - l_2 & \quad & - d + l_2 + x + y + \s^2_x & \quad & 2z - 2x \\
& \quad & - d + l_1 + 2x + \s^2_y & \quad & y  - z \\
& \quad & - d + l_1 + x + y + \s^2_z & \quad & z - y
\ea.
\]
We next examine the case when $\Z$ is a triple point.
We may assume that $\Z$ is supported on $(0, 0, 1)$, the other cases being obtained by a permutation of variables.
Consider thus morphisms of the form
\[
{
\mathversion{bold}
\boldsymbol{\eta (l_1, l_2, d)} = \left[
\ba{cccc}
l_1 & l_2 & 0 & 0 \\
c_{11} d/l_2 Y^2 & c_{12} d/l_1 Y^2 & X & 0 \\
c_{21} d/l_2 X^2 & c_{22} d/l_1 X^2 & 0 & Y \\
c_{31} d/l_2 XY & c_{32} d/l_1 XY & Y & X
\ea
\right].
}
\]
We have
\[
\ba{lll}
u_1 = d - l_2 & \qquad & v_1 = - d + l_1 + l_2 \\
u_2 = d - l_1 & \qquad & v_2 = - 2y \\
u_3 = x + 2y & \qquad & v_3 = - 2x \\
u_4 = 2x + y & \qquad & v_4 = - x - y
\ea.
\]
The torus $T$ acts on $\CC s_1$ with weight $x + 2y - d + l_1 + l_2$ and on $\CC s_2$ with weight $2x + y - d + l_1 + l_2$.
The list of weights for the action of $T$ on $\T_{[\eta]} \M$ is obtained by subtracting the list
$\{ \chi_0, \ - d + l_1 + l_2 + x + y + z \}$ from the list
\[
\ba{lllll}
d - l_1 - l_2 - x - 2y & \quad & - d + l_2 + 2y + \s^2_{l_1} & \quad & z - x \\
d - l_1 - l_2 - 2x - y & \quad & - d + l_1 + 2x + \s^2_{l_2} & \quad & z - x \\
l_3 - l_1 & \quad & - d + l_2 + x + y + \s^2_y & \quad & z - y \\
l_3 - l_2 & \quad & - d + l_2 + 2x + \s^2_y & \quad & z - y \\
& \quad & - d + l_1 + x + y + \s^2_x & \quad & x - 2y + z \\
& \quad & - d + l_1 + 2y + \s^2_x & \quad & y - 2x + z
\ea.
\]
Finally, we assume that $(q_1, q_2, q_3) = (X^2, XY, XZ)$, that is, we consider morphisms of the form
\[
{
\mathversion{bold}
\boldsymbol{\th (l_1, l_2, d)} = \left[
\ba{cccc}
l_1 & l_2 & 0 & 0 \\
c_{11} d/l_2 XY & c_{12} d/l_1 XY & X & 0 \\
c_{21} d/l_2 XZ & c_{22} d/l_1 XZ & 0 & X \\
c_{31} d/l_2 X^2 & c_{32} d/l_1 X^2 & Y & Z
\ea
\right].
}
\]
We have
\[
\ba{lll}
u_1 = d - l_2 & \qquad & v_1 = - d + l_1 + l_2 \\
u_2 = d - l_1 & \qquad & v_2 = - x - y \\
u_3 = 2x + y & \qquad & v_3 = - x - z \\
u_4 = 2x + z & \qquad & v_4 = - 2x
\ea.
\]
The torus $T$ acts on $\CC s_1$ with weight $2x + y - d + l_1 + l_2$ and on $\CC s_2$ with weight $2x + z - d + l_1 + l_2$.
The list of weights for the action of $T$ on $\T_{[\th]} \M$ is obtained by subtracting the list
$\{ \chi_0, \ - d + l_1 + l_2 + 3x \}$ from the list
\[
\ba{lllll}
d - l_1 - l_2 - 2x - y & \quad & - d + l_2 + x + y + \s^2_x & \quad & y - x \\
d - l_1 - l_2 - 2x - z & \quad  & - d + l_2 + 2x + \s^2_z & \quad & y - x \\
l_3 - l_1 & \quad & - d + l_1 + x + y + \s^2_{l_2} & \quad & z - x \\
l_3 - l_2 & \quad & - d + l_2 + x + z + \s^2_{l_1} & \quad & z - x \\
& \quad & - d + l_1 + 2x + \s^2_y & \quad & 2z - x - y \\
& \quad & - d + l_1 + x + z + \s^2_x & \quad & 2y - x - z
\ea.
\]
All fixed points in $\M_1$ can be obtained from the points $\e$, $\z$, $\eta$, $\th$ above by a permutation of variables.
We have thus found the $T$-representation for all tangent spaces at fixed points in $\M_1$.
To complete the picture, we need to determine which points of $\M_1^{T}$ lie in the closure of
\[
(\M \setminus \overline{\M}_1)^{T} = \M_0^{T}.
\]
This can be done using Proposition \ref{3.2.1}.
For instance, $T$ acts trivially on $N_{[\f]}$ for $\f = \e(X^2 Y^2 Z)$.
This shows that $[\e(X^2 Y^2 Z)]$ belongs to the closure of a surface in $\M_0^{T}$.
There are only three such surfaces: $S$ and two others obtained by interchanging $X$ and $Z$, respectively, $Y$ and $Z$.
The sheaves giving points in $S$ have support given by the equation $\{ X^2 Y^2 Z = 0 \}$.
We deduce that $[\e(X^2 Y^2 Z)]$ belongs to $S$ and that
$[\e(X Y^2 Z^2)]$, $[\e(X^2 Y Z^2)]$ belong to the other two surfaces.
For the remaining points in column five of Table 2 below, we have $\dim N_{\chi_0} = 1$,
which shows that they belong to the closure of affine lines in $\M_0^{T}$.

\begin{table}[!hpt]{Table 2. Fixed points in $\M_{1}$.}
\begin{center}
\begin{tabular}{| c | c | l | c | c |}
\hline
$(q_1, q_2, q_3)$ & $(l_1, l_2)$ & equation of support & affine lines & limit sheaves
\\
\hline
$(XY, XZ, YZ)$ & $(X, Y)$ & $\Sigma^5 \setminus \{ X^5, Y^5, Z^5, XZ^4, YZ^4 \}$ & $XYZ^3$ & $X^2 Y^2 Z$
\\
\cline{2-5}
& $(X, Z)$ & $\Sigma^5 \setminus \{ X^5, Y^5, Z^5, XY^4, Y^4 Z \}$ & $XY^3 Z$ & $X^2 YZ^2$
\\
\cline{2-5}
& $(Y, Z)$ & $\Sigma^5 \setminus \{ X^5, Y^5, Z^5, X^4 Y, X^4 Z \}$ & $X^3 YZ$ & $XY^2 Z^2$
\\
\hline
$(X^2, XY, YZ)$ & $(X, Y)$ & $\Sigma^5 \setminus \{ Y^5, Z^5, XZ^4, X^2 Z^3, YZ^4 \}$ & $X^2 YZ^2$ &
\begin{tabular}{l} $X^3 Y^2$ \\ $X^2 Y^2 Z$ \end{tabular}
\\
\cline{2-5}
& $(X, Z)$ & $\Sigma^5 \setminus \{ Y^5, Z^5, XY^4, XZ^4, Y^4 Z \}$ & $XY^3 Z$ &
\begin{tabular}{l} $X^3 YZ$ \\ $X^2 YZ^2$ \end{tabular}
\\
\cline{2-5}
& $(Y, Z)$ & $\Sigma^5 \setminus \{ X^5, Y^5, Z^5, XZ^4 \}$ & &
\begin{tabular}{l} $X^2 Y^2 Z$ \\ $XY^2 Z^2$ \end{tabular}
\\
\hline
$(X^2, XY, Y^2)$ & $(X, Y)$ &
\begin{tabular}{l} $\Sigma^5 \setminus \{ Z^5, Z^4 X, Z^4 Y, Z^3 X^2$, \\ $\phantom{\Sigma^5 \setminus \{} Z^3 XY, Z^3 Y^2 \}$ \end{tabular} &
\begin{tabular}{l} $XY^2 Z^2$ \\ $X^2 YZ^2$ \end{tabular} &
\begin{tabular}{l} $X^3 Y^2$ \\ $X^2 Y^3$ \end{tabular}
\\
\cline{2-5}
& $(X, Z)$ & $\Sigma^5 \setminus \{ Y^5, Z^5, XZ^4, YZ^4 \}$ & &
\begin{tabular}{l} $X^2 Y^2 Z$ \\ $X^3 YZ$ \end{tabular}
\\
\cline{2-5}
& $(Y, Z)$ & $\Sigma^5 \setminus \{ X^5, Z^5, XZ^4, YZ^4 \}$ & &
\begin{tabular}{l} $X^2 Y^2 Z$ \\ $XY^3 Z$ \end{tabular}
\\
\hline
$(X^2, XY, XZ)$ & $(X, Y)$ &
\begin{tabular}{l} $\Sigma^5 \setminus \{ Y^5, Z^5, XZ^4, YZ^4$, \\ $\phantom{\Sigma^5 \setminus \{} Y^4 Z, Y^2 Z^3, Y^3 Z^2 \}$ \end{tabular} &
\begin{tabular}{l} $X^2 YZ^2$ \\ $XY^3 Z$ \\ $XY^2 Z^2$ \end{tabular} &
\begin{tabular}{l} $X^3 YZ$ \\ $X^3 Y^2$ \end{tabular}
\\
\cline{2-5}
& $(X, Z)$ &
\begin{tabular}{l} $\Sigma^5 \setminus \{ Y^5, Z^5, XY^4, YZ^4$, \\ $\phantom{\Sigma^5 \setminus \{} Y^4 Z, Y^2 Z^3, Y^3 Z^2 \}$ \end{tabular} &
\begin{tabular}{l} $X^2 Y^2 Z$ \\ $XYZ^3$ \\ $XY^2 Z^2$ \end{tabular} &
\begin{tabular}{l} $X^3 YZ$ \\ $X^3 Z^2$ \end{tabular}
\\
\cline{2-5}
& $(Y, Z)$ &
\begin{tabular}{l} $\Sigma^5 \setminus \{ X^5, Y^5, Z^5, YZ^4$, \\ $\phantom{\Sigma^5 \setminus \{} Y^2 Z^3, Y^3 Z^2, Y^4 Z \}$ \end{tabular} &
\begin{tabular}{l} $XYZ^3$ \\ $XY^2 Z^2$ \\ $XY^3 Z$ \end{tabular} &
\begin{tabular}{l} $X^2 Y^2 Z$ \\ $X^2 YZ^2$ \end{tabular}
\\
\hline
\end{tabular}
\end{center}
\end{table}


\subsection{Fixed points in $\M_2$}
\label{3.3}

The list of weights for the action of $T$ on $\T_{\f} W_2$ given by formula (3.1.1) is represented by the array
\[
\ba{rrrrr}
- v_1 - u_1 + \s^1 & \quad & - v_1 - u_2 + \s^1 & \quad & - v_1 - u_3 + \s^1 \\
- v_2 - u_1 + \s^1 & \quad & - v_2 - u_2 + \s^1 & \quad & - v_2 - u_3 + \s^1 \\
- v_3 - u_1 + \s^3 & \quad & - v_3 - u_2 + \s^3 & \quad & - v_3 - u_3 + \s^3
\ea.
\]
The list of weights for the action of $T$ on $\T_{\f} (G_2 \f)$ given by formula (3.1.2) is represented by the array
\[
\ba{lllllrrrr}
\chi_0 & \quad & u_1 - u_2 & \quad & u_1 - u_3 & \quad & \chi_0 & \quad & - v_1 + v_2 \\
u_2 - u_1 & \quad & \chi_0 & \quad & u_2 - u_3 & \quad & - v_2 + v_1 & \quad & \chi_0 \\
u_3 - u_1 & \quad & u_3 - u_2 & \quad & \chi_0 & \quad & - v_3 + v_1 + \s^2 & \quad & - v_3 + v_2 + \s^2
\ea.
\]

\begin{proposition}
\label{3.3.1}
The normal space $N_{[\F]}$ to $\M_2$ at $[\F]$ can be identified with
\[
\H^0(\F(-1))^* \tensor \H^1(\F(-1)).
\]
The torus $T$ acts on $N_{[\F]}$ with weights $u_i + v_3 - x - y - z$, $i = 1, 2, 3$.
\end{proposition}

\begin{proof}
We apply \cite[Proposition 6.2]{choi_maican} to the sheaf $\F^\D(1)$, which gives a point in $\M_{\PP^2}(5,2)$.
We have a resolution
\[
0 \lra \O(-3) \oplus 2\O(-1) \stackrel{\f^\TR}{\lra} 3\O \lra \F^\D(1) \lra 0.
\]
Let $N^\D$ denote the normal space to $\M_2^\D \subset \M_{\PP^2}(5,2)$ at $[\F^\D(1)]$.
As $T$-modules, $N$ and $N^\D$ are isomorphic.
This proves the second part of the proposition.
Moreover,
\begin{align*}
N \isom \N^\D & \isom \H^0(\F^\D(1))^* \tensor \H^1(\F^\D(1)) \qquad \text{by \cite[Proposition 6.2]{choi_maican}} \\
& \isom \H^1(\F(-1)) \tensor \H^2(\omega_{\PP^2})^* \tensor \H^0(\F(-1))^* \tensor \H^2(\omega_{\PP^2}) \quad
\text{by Serre Duality} \\
& \isom \H^1(\F(-1)) \tensor \H^0(\F(-1))^*.
\qedhere
\end{align*}
\end{proof}

\noindent
Recall from Proposition \ref{2.4.1} that an irreducible component of $\M_2^{T}$ is uniquely determined by
$q_1$, $q_2$, $q_3$ and $d$.
The ideal $(q_1, q_2, q_3)$ defines a $T$-invariant zero-dimensional scheme $\Z$ of length $3$ that is not contained
in a line or is of the form $(lX, lY, lZ)$, $l \in \{ X, Y, Z \}$.
Assume firstly that $\f$ has the form
\[
{
\mathversion{bold}
\boldsymbol{\iota(d)} = \left[
\ba{ccc}
X & Y & 0 \\
X & 0 & Z \\
c_1 d/YZ & c_2 d/XZ & c_3 d/XY
\ea
\right]
}
\]
corresponding to the case when $\Z$ consists of three distinct points. Clearly, we have
\[
\ba{lll}
u_1 = x & \qquad & v_1 = 0 \\
u_2 = y & \qquad & v_2 = 0 \\
u_3 = z & \qquad & v_3 = d - x - y - z
\ea.
\]
The list of weights for the action of $T$ on $\T_{[\iota(d)]} \M$ is obtained by removing an extra copy of $\chi_0$
from the table
\[
\ba{lllll}
\s^0 & \quad & d - x - 2y - 2z & \quad & - d + y + z + \s^3_x \\
& & d - 2x - y - 2z & \quad & - d + x + z + \s^3_y \\
& & d - 2x - 2y - z & \quad & - d + x + y + \s^3
\ea.
\]
We next examine the case when $\Z$ is the union of a double point and a simple point.
Consider thus morphisms of the form
\[
{
\mathversion{bold}
\boldsymbol{\kappa (d)} = \left[
\ba{ccc}
Y & 0 & X \\
0 & X & Z \\
c_1 d/X^2 & c_2 d/YZ & c_3 d/XY
\ea
\right].
}
\]
We have
\[
\ba{lll}
u_1 = y - x & \qquad & v_1 = x \\
u_2 = x - z & \qquad & v_2 = z \\
u_3 = 0 & \qquad & v_3 = d - x - y
\ea.
\]
The list of weights for the action of $T$ on $\T_{[\kappa(d)]} \M$ is obtained by removing an extra copy of $\chi_0$
from the array
\[
\ba{lllllll}
x - y & \quad & y - z & \quad & d - 3x - y - z & \quad & - d + y + z + \s^3_x \\
y - x & \quad & z - y & \quad & d - x - 2y - 2z & \quad & - d + 2x + \s^3_y \\
z - x & \quad & 2z - 2x & \quad & d - 2x - 2y - z & \quad & - d + x + y + \s^3
\ea.
\]
We assume now that $\Z$ is a triple point supported at $(0, 0, 1)$,
that is, we consider morphisms of the form
\[
{
\mathversion{bold}
\boldsymbol{\lambda (d)} = \left[
\ba{ccc}
X & 0 & Y \\
0 & Y & X \\
c_1 d/Y^2 & c_2 d/X^2 & c_3 d/XY
\ea
\right].
}
\]
We have
\[
\ba{lll}
u_1 = x - y & \qquad & v_1 = y \\
u_2 = y - x & \qquad & v_2 = x \\
u_3 = 0 & \qquad & v_3 = d - x - y
\ea.
\]
The list of weights for the action of $T$ on $\T_{[\lambda(d)]} \M$ is obtained by removing an extra copy of $\chi_0$
from the array
\[
\ba{lllllll}
z - x & \quad & z - y & \quad & d - 3x - y - z & \quad & - d + 2x + \s^3_y \\
z - x & \quad & z - y & \quad & d - x - 3y - z & \quad & - d + x + y + \s^3 \\
z - 2x + y & \quad & z + x - 2y & \quad & d - 2x - 2y - z & \quad & - d + 2y + \s^3_x
\ea.
\]
Finally, we assume that $(q_1, q_2, q_3) = (X^2, XY, XZ)$, that is, we consider morphisms of the form
\[
{
\mathversion{bold}
\boldsymbol{\mu (d)} = \left[
\ba{ccc}
X & 0 & Y \\
0 & X & Z \\
c_1 d/XY & c_2 d/XZ & c_3 d/X^2
\ea
\right].
}
\]
We have
\[
\ba{lll}
u_1 = x - y & \qquad & v_1 = y \\
u_2 = x - z  & \qquad & v_2 = z \\
u_3 = 0 & \qquad & v_3 = d - 2x
\ea.
\]
The list of weights for the action of $T$ on $\T_{[\mu(d)]} \M$ is obtained by removing an extra copy of $\chi_0$
from the table
\[
\ba{lllllll}
z - x & \quad & y - x & \quad & d - 2x - 2y - z & \quad & - d + x + y + \s^3_x \\
z - x & \quad & y - x & \quad & d - 2x - y - 2z & \quad & - d + 2x + \s^3 \\
2z - x - y & \quad & 2y - x - z & \quad & d - 3x - y - z & \quad & - d + x + z + \s^3_x
\ea.
\]
All fixed points in $\M_2$ can be obtained from the points $\iota$, $\kappa$, $\lambda$, $\mu$ above by a permutation of variables.
We have thus found the $T$-representation for all tangent spaces at fixed points in $\M_2$.
To complete the picture, we need to determine which points of $\M_2^{T}$ lie in the closure of
\[
(\M \setminus \overline{\M}_2)^{T} = (\M_0 \cup \M_1)^{T}.
\]
This can be done using Proposition \ref{3.3.1}.
For all points in column four of Table 3 below, we have $\dim N_{\chi_0} = 1$.
This shows that these points lie in the closure of affine lines in $(\M_0 \cup \M_1)^{T}$.

\begin{table}[!hpt]{Table 3. Fixed points in $\M_{2}$.}
\begin{center}
\begin{tabular}{| c | l | c | c |}
\hline
$(q_1, q_2, q_3)$ & equation of support & affine lines & limit sheaves
\\
\hline
$(XY, XZ, YZ)$ & $\Sigma^5 \setminus \{ X^5, Y^5, Z^5 \}$ & &
\begin{tabular}{l} $X^2 Y^2 Z$ \\ $XY^2 Z^2$ \\ $X^2 YZ^2$ \end{tabular}
\\
\hline
$(X^2, XY, YZ)$ & $\Sigma^5 \setminus \{ Y^5, Z^5, XZ^4 \}$ & &
\begin{tabular}{l} $X^2 Y^2 Z$ \\ $XY^2 Z^2$ \\ $X^3 YZ$ \end{tabular}
\\
\hline
$(X^2, XY, Y^2)$ & $\Sigma^5 \setminus \{ Z^5, Z^4 X, Z^4 Y \}$ & &
\begin{tabular}{l} $X^2 Y^2 Z$ \\ $XY^3 Z$ \\ $X^3 YZ$ \end{tabular}
\\
\hline
$(X^2, XY, XZ)$ & $\Sigma^5 \setminus \{ Y^5, Y^4 Z, Y^3 Z^2, Y^2 Z^3, YZ^4, Z^5 \}$ &
\begin{tabular}{l} $XYZ^3$ \\ $XY^2 Z^2$ \\ $XY^3 Z$ \end{tabular} &
\begin{tabular}{l} $X^2 Y^2 Z$ \\ $X^2 YZ^2$ \\ $X^3 YZ$ \end{tabular}
\\
\hline
\end{tabular}
\end{center}
\end{table}


\subsection{Fixed points in $\M_3$}
\label{3.4}

The list of weights for the action of $T$ on $\T_{\f} W_3$ given by formula (3.1.1) is represented by the array
\[
\ba{rrr}
- v_1 - u_1 + \s^3 & \quad & - v_1 - u_2 + \s^1 \\
- v_2 - u_1 + \s^4 & \quad & - v_2 - u_2 + \s^2
\ea.
\]
The list of weights for the action of $T$ on $\T_{\f} (G_3 \f)$ given by formula (3.1.2) is represented by the array
\[
\ba{lllll}
\chi_0 & \quad & \phantom{-} \chi_0 \\
u_2 - u_1 + \s^2 & \quad & - v_2 + v_1 & \quad & \chi_0
\ea.
\]

\begin{proposition}
\label{3.4.1}
The normal space $N_{[\F]}$ to $\M_3$ at $[\F]$ can be identified with
\[
\H^0(\F)^* \tensor \H^1(\F).
\]
The torus $T$ acts on $N_{[\F]}$ with weights
\[
u_1 + v_1 - x - y - z, \quad u_1 + v_2 - 2x - y - z, \quad u_1 + v_2 - x - 2y - z, \quad u_1 + v_2 - x - y - 2z.
\]
\end{proposition}

\begin{proof}
The argument is analogous to the argument at \cite[Proposition 6.2]{choi_maican}.
Let $\k \colon \H^0(\F) \tensor \O \to \F$ be the canonical morphism and let $\K = \Ker(\k)$.
Applying the snake lemma to the diagram
\[
\xymatrix
{
& & 0 \ar[d] & 0 \ar[d] \\
& & \Omega^1(1) \ar[d] & \K \ar[d] \\
& 0 \ar[r] & 4\O \ar[r] \ar[d] & \H^0(\F) \tensor \O \ar[r] \ar[d]^-{\k} & 0 \\
0 \ar[r] & \O(-3) \oplus \O(-1) \ar[r]^-{\f} & \O \oplus \O(1) \ar[r] \ar[d] & \F \ar[r] \ar[d] & 0 \\
& & 0 & 0
}
\]
we get the exact sequence
\[
0 \lra \Omega^1(1) \lra \K \lra \O(-3) \oplus \O(-1) \lra 0.
\]
We claim that $\Ext^1(\K, \F) = 0$.
This follows from the long $\Ext(\_, \F)$-sequence associated to the above short exact sequence
and from the vanishing of
\[
\Ext^1(\O(-3), \F), \qquad \Ext^1(\O(-1), \F), \qquad \Ext^1(\Omega^1(1), \F).
\]
To see that the last group vanishes we apply the long $\Ext(\Omega^1(1),\_)$-sequence to the short
exact sequence expressing $\F$ as the cokernel of $\f$ and we use the vanishing of
\[
\Ext^1(\Omega^1(1),\O), \ \
\Ext^1(\Omega^1(1),\O(1)), \ \
\Ext^2(\Omega^1(1),\O(-3)), \ \
\Ext^2(\Omega^1(1),\O(-1)).
\]
Applying the long $\Ext(\_, \F)$-sequence to the exact sequence
\[
0 \lra \K \lra \H^0(\F) \tensor \O \stackrel{\k}{\lra} \F \lra 0
\]
we obtain a surjective map $\e \colon \Ext^1(\F, \F) \to \H^0(\F)^* \tensor \H^1(\F)$.
Consider an extension sheaf
\[
0 \lra \F \lra \E \stackrel{\pi}{\lra} \F \lra 0.
\]
Assume that the extension class of $\E$ belongs to $\Ker(\e)$.
Then, as in \cite[Proposition 6.2]{choi_maican}, we deduce that $\H^0(\pi)$ has a splitting.
Thus, $\h^0(\E) = 8$ and, from Lemma \ref{3.4.2} below, we have $\h^0(\E(-1)) = 2$.
It becomes clear now that we can apply the horseshoe lemma to the above extension in order to obtain
a resolution of the form
\[
0 \to (\O(-3) \oplus \O(-1)) \oplus (\O(-3) \oplus \O(-1)) \stackrel{\psi}{\to} (\O \oplus \O(1)) \oplus (\O \oplus \O(1)) \to
\E \to 0,
\]
\[
\psi = \left[
\ba{cc}
\f & w \\
0 & \f
\ea
\right].
\]
Thus, $\E$ gives a tangent vector to $\M_3$, namely the image of $w$ in $W_3/\T_e G_3 \isom \T_{[\F]} \M_3$.
Since $\operatorname{Ker}(\e)$ is contained in $\T_{[\F]} \M_3$ and both spaces have dimension $22$,
we conclude that $\operatorname{Ker}(\e) = \T_{[\F]} \M_3$.
This proves the first part of the proposition.

To determine the action of $T$ on $N$ consider the commutative diagram
\[
\xymatrix
{
0 \ar[r] & \O(-3) \oplus \O(-1) \ar[r]^-{\f} \ar[d]^-{u(t)^{-1}} & \O \oplus \O(1) \ar[r] \ar[d]^-{v(t)} & \F \ar[r] \ar[d]^-{\y(t)} & 0 \\
0 \ar[r] & \O(-3) \oplus \O(-1) \ar[r]^-{t \f} & \O \oplus \O(1) \ar[r] & t \F \ar[r] & 0
}.
\]
The action of $T$ on $\H^0(\F)$ is given by $t s = \y(t)^{-1} \mu_{t^{-1}}^* (s)$.
Under the identification $\CC \oplus V^* \isom \H^0(\F)$, we have $t s = v(t)^{-1} \mu_{t^{-1}}^* (s)$.
Thus, $T$ acts on $\H^0(\F)$ with weights $- v_1$, $- v_2 + x$, $- v_2 + y$, $- v_2 + z$.
As in \cite[Proposition 6.2]{choi_maican}, $T$ acts on $\H^1(\F)$ with weight $u_1 - x - y - z$.
\end{proof}

\begin{lemma}
\label{3.4.2}
Assume that $\F$ gives a point in $\M_3$. Let $\E$ be an extension of $\F$ by $\F$.
If $\h^0(\E)=8$, then $\h^0(\E(-1))=2$.
\end{lemma}

\begin{proof}
Assume, instead, that $\h^0(\E(-1))=1$.
Denote $m = \h^1(\E \tensor \Omega^1(1))$.
The $\operatorname{E}^1$-level of the Beilinson spectral sequence \cite[(2.2.3)]{dedicata}
converging to $\E$ reads
\[
\xymatrix
{
5\O(-2) \ar[r]^-{\f_1} & m\O(-1) \ar[r]^-{\f_2} & 2\O \\
\O(-2) \ar[r]^-{\f_3} & (m+2) \O(-1) \ar[r]^-{\f_4} & 8\O
}.
\]
Note that $m \ge 5$ because $\f_2$ is surjective.
Note that $m \le 7$ because $\E$ maps surjectively to $\Ker(\f_2)/\Im(\f_1)$.
If $m=7$, then the first row above is a monad, so its cohomology has slope $-2/3$,
destabilising $\E$. Thus, $m=5$ or $6$.

Assume, firstly, that $m=5$. Denote $\G = \E^\D(1)$.
The Beilinson free monad \cite[(2.2.1)]{dedicata} yields a resolution
\[
0 \lra 2\O(-2) \stackrel{\psi}{\lra} 8\O(-2) \oplus 5\O(-1) \stackrel{\f}{\lra} \Omega^1 \oplus 4\O(-1) \oplus 5\O \lra \G \lra 0
\]
in which $\f_{12}=0$, $\f_{22}=0$. The map $5\O \to \G$ is an isomorphism on global sections and $\G$ has no zero-dimensional
torsion.
This shows, as in the proof of \cite[Proposition 3.1.3]{illinois}, that $\psi_{21}$ has one of the following canonical forms:
\[
\left[
\ba{cc}
0 & 0 \\
0 & 0 \\
X & R \\
Y & S \\
Z & T
\ea
\right], \qquad \qquad \left[
\ba{cc}
0 & 0 \\
X & 0 \\
Y & R \\
Z & S \\
0 & T
\ea
\right], \qquad \qquad \left[
\ba{cc}
X & 0 \\
Y & 0 \\
Z & R \\
0 & S \\
0 & T
\ea
\right].
\]
Here $\{ R, S, T \}$ is a basis of $V^*$.
At \cite[Proposition 3.1.3]{illinois} it is shown how each of these forms leads to a contradiction.

Assume next that $m=6$. 
The Beilinson free monad \cite[(2.2.1)]{dedicata} yields the resolution
\[
0 \lra \O(-2) \lra 5\O(-2) \oplus 8\O(-1) \lra 2\Omega^1 \oplus 8\O \lra \E \lra 0,
\]
hence the resolution
\[
0 \lra \O(-2) \lra 2\O(-3) \oplus 5\O(-2) \oplus 8\O(-1) \stackrel{\f}{\lra} 6\O(-2) \oplus 8 \O \lra \E \lra 0.
\]
As in the argument at \cite[Proposition 2.1.4]{illinois}, the rank of $\f_{12}$ is maximal,
otherwise $\E$ would map surjectively to the cokernel of a morphism $2\O(-3) \to 2\O(-2)$,
violating semi-stability. We arrive at the exact sequence
\[
0 \lra \O(-2) \stackrel{\psi}{\lra} 2\O(-3) \oplus 8\O(-1) \lra \O(-2) \oplus 8\O \lra \E \lra 0.
\]
The map $8\O \to \E$ is an isomorphism on global sections.
It follows, as in the argument at \cite[Proposition 2.1.4]{illinois}, that $\Coker(\psi_{12}) \isom 5\O(-1) \oplus \Omega^1(1)$.
Combining the resolution
\[
0 \lra 2\O(-3) \oplus 5\O(-1) \oplus \Omega^1(1) \lra \O(-2) \oplus 8\O \lra \E \lra 0
\]
with the Euler sequence we obtain the resolution
\[
0 \lra 2\O(-3) \oplus 5\O(-1) \oplus 3\O \stackrel{\f}{\lra} \O(-2) \oplus 8\O \oplus \O(1) \lra \E \lra 0.
\]
As before, $\f_{23}$ has maximal rank, otherwise $\E$ would map surjectively to the cokernel of a morphism
$2\O(-3) \oplus 5\O(-1) \to \O(-2) \oplus 6\O$, contradicting semi-stability.
Canceling $3\O$ we obtain a resolution of $\E$ that fits into a commutative diagram
\[
\xymatrix
{
0 \ar[r] & \O(-3) \oplus \O(-1) \ar[r] \ar[d]^-{\b} & \O \oplus \O(1) \ar[r] \ar[d]^-{\a} & \F \ar[r] \ar[d] & 0 \\
0 \ar[r] & 2\O(-3) \oplus 5\O(-1) \ar[r] & \O(-2) \oplus 5\O \oplus \O(1) \ar[r] & \E \ar[r] & 0
}.
\]
The map $\F \to \E$ above is the inclusion given in the hypothesis of the lemma, so its cokernel is $\F$.
From the snake lemma we obtain the exact sequence
\[
0 \lra \Coker(\b) \lra \Coker(\a) \lra \F \lra 0.
\]
Since $\a(-1)$ and $\a$ are injective on global sections, we see that $\a$ is injective, hence $\Coker(\a) \isom \O(-2) \oplus 4\O$.
From the above exact sequence we see that $\h^1(\Coker(\b)(-1)) = 1$.
On the other hand, if $\b_{11} \neq 0$, then $\h^1(\Coker(\b)(-1)) = 0$.
If $\b_{11} = 0$, then $\h^1(\Coker(\b)(-1)) = 3$.
We have arrived at a contradiction.
Our original assumption that $\h^0(\E(-1))=1$ must be wrong. This proves the lemma.
\end{proof}

\noindent
Recall from Section \ref{2.5} that $\M_3^{T}$ consists of isolated points of the form $\O_Q(-\Y)(2)$.
They are cokernels of morphisms in $W_3$ of the form
\[
{
\mathversion{bold}
\boldsymbol{\nu (l, q, d)} = \left[
\ba{cc}
c_1 d/q & l \\
c_2 d/l & q
\ea
\right],
}
\]
where $l \in \{ X, Y, Z \}$, $q$ is a quadratic monomial, $l$ does not divide $q$, $d$ is a monomial of degree $5$
in the ideal $(l, q)$, and $c_1, c_2 \in \CC$.
We have
\[
\ba{lll}
u_1 = d - q - l & \qquad & v_1 = l \\
u_2 = 0 & \qquad  & v_2 = q
\ea.
\]
All fixed points in $\M_3$ can be obtained from the points $\nu(X, q, d)$ by swapping $X$ and $Y$, or by swapping $X$ and $Z$.
Thus, in order to describe the $T$-representation for all tangent spaces at fixed points in $\M_3$,
it is enough to give the list of weights for the action of $T$ on $\T_{[\nu(X, q, d)]} \M$.
This is
\[
\ba{lllll}
d - q - x - y - z & \quad & - d + x + (\s^4 \setminus \{ d - x \}) & \quad & - x + y \\
d - 3x - y - z & \quad & - d + q + \s^3_x & \quad & - x + z \\
d - 2x - 2y - z & \quad & & \quad & - q + (\s^2_x \setminus \{ q \}) \\
d - 2x - y - 2z & \quad & & \quad &
\ea.
\]
To give a complete picture for the torus fixed locus in $\M_3$ we need to determine which points of $\M_3^T$
lie in the closure of
\[
(\M \setminus \M_3)^T = (\M_0 \cup \M_1 \cup \M_2)^T.
\]
This can be done using Proposition \ref{3.4.1}.
For all points represented in column 4 of Table 4 below we have $\dim N_{\chi_0} = 1$.
This shows that these points lie in the closure of affine lines in $(\M_0 \cup \M_1 \cup \M_2)^T$.

\begin{table}[!hpt]{Table 4. Fixed points in $\M_{3}$.}
\begin{center}
\begin{tabular}{| c | c | l | l |}
\hline
$l$ & $q$ & $d$ & limit sheaves \\
\hline
$X$ & $Y^2$ & $\Sigma^5 \setminus \{ Z^5, YZ^4 \}$ & $Y^3 Z$, $X^2 YZ$, $XY^2 Z$, $XYZ^2$ \\
\cline{2-4}
& $Z^2$ & $\Sigma^5 \setminus \{ Y^5, Y^4 Z \}$ & $YZ^3$, $X^2 YZ$, $XY^2 Z$, $XYZ^2$ \\
\cline{2-4}
& $YZ$ & $\Sigma^5 \setminus \{ Y^5, Z^5 \}$ & $Y^2 Z^2$, $X^2 YZ$, $XY^2 Z$, $XYZ^2$ \\
\hline
$Y$ & $X^2$ & $\Sigma^5 \setminus \{ Z^5, XZ^4 \}$ & $X^3 Z$, $X^2 YZ$, $XY^2 Z$, $XYZ^2$ \\
\cline{2-4}
& $Z^2$ & $\Sigma^5 \setminus \{ X^5, X^4 Z \}$ & $XZ^3$, $X^2 YZ$, $XY^2 Z$, $XYZ^2$ \\
\cline{2-4}
& $XZ$ & $\Sigma^5 \setminus \{ X^5, Z^5 \}$ & $X^2 Z^2$, $X^2 YZ$, $XY^2 Z$, $XYZ^2$ \\
\hline
$Z$ & $Y^2$ & $\Sigma^5 \setminus \{ X^5, X^4 Y\}$ & $XY^3$, $X^2 YZ$, $XY^2 Z$, $XYZ^2$ \\
\cline{2-4}
& $X^2$ & $\Sigma^5 \setminus \{ Y^5, XY^4 \}$ & $X^3 Y$, $X^2 YZ$, $XY^2 Z$, $XYZ^2$ \\
\cline{2-4}
& $XY$ & $\Sigma^5 \setminus \{ X^5, Y^5\}$ & $X^2 Y^2$, $X^2 YZ$, $XY^2 Z$, $XYZ^2$ \\
\hline
\end{tabular}
\end{center}
\end{table}


\subsection{Proof of the main theorem}

We are now ready to prove the theorem announced in the introduction.
We first recall some general facts. It is known that every irreducible component $X$ of $\M^T$ is a smooth subvariety of $\M$.
The irreducible components of $\M^T$ are disjoint. For every $x \in X$ we have $\dim (\T_x \M)_{\chi_0} = \dim X$.
Given points $x, y \in X$, the tangent spaces $\T_x \M$ and $\T_y \M$ are isomorphic as $T$-modules.
Inspecting Tables 2, 3, and 4, we see that the only fixed points $[\f]$ outside $\M_0$ for which $\dim (\T_{[\f]} \M)_{\chi_0} = 2$
are $[\e(X^2 Y^2 Z)]$, $[\e(XY^2 Z^2)]$, and $[\e(X^2 YZ^2)]$. Therefore, these are the only points of $\M \setminus \M_0$
that belong to an irreducible component of $\M^T$ of dimension $2$. All other points in $(\M \setminus \M_0)^T$
belong to an irreducible component of the fixed locus of dimension $0$ or $1$.
In view of Section \ref{2.2.2} and of the comments at the end of Section \ref{3.2},
we conclude that there are only three irreducible components of $\M^T$ of dimension greater than $1$,
namely $S$ and two other surfaces obtained from $S$ by interchanging $X$ and $Z$, respectively, by interchanging $Y$ and $Z$.
As mentioned at Section \ref{2.2.2}, $S$ can be obtained from $\PP^1 \times \PP^1$ by blowing up three points on the diagonal
and then blowing down the strict transform of the diagonal.
The Poincar\'e polynomial of $S$ is $P(x) = x^4 + 4x^2 + 1$ and its Hodge numbers satisfy the relation $h^{pq} = 0$ if $p \neq q$.
A further examination of Tables 1, 2, 3, and 4, proves the first part of the theorem concerning the structure of $\M^T$.
Beside the three surfaces, there are $1329$ isolated points and $174$ irreducible components isomorphic to $\PP^1$.
Let $\Pi$ be the set of points, let $\Lambda$ be the set of lines, let $\Sigma$  be the set of surfaces in $\M^T$.
Let $\Xi$ be the set of irreducible components of $\M^T$.

Let $\lambda(t) = (t^{n_0},\ t^{n_1},\ t^{n_2})$ be a one-parameter subgroup of $T$ that is not orthogonal to any non-zero
character $\chi \in \chi^*(T)$ appearing in the weight decomposition of some $\T_x \M$, $x \in \M^T$.
Inspecting the tables of characters from Section \ref{section_3}, we see that the set of characters $\chi \neq \chi_0$ for which
$\dim (\T_x \M)_{\chi} \neq 0$ for some $x \in \M^T$ is contained in the set
\[
\{ ix + jy +kz, \quad -6 \le i, j, k \le 6 \}.
\]
We can thus choose $\lambda(t) = (1, t, t^7)$. We consider the $\CC^*$ action on $\M$ induced by $\lambda$.
Given $X \in \Xi$ and $x \in X$, we denote by $p(X)$ the dimension of the subspace
of $\T_x \M$ on which $\CC^*$ acts with positive weights. In other words, $p(X)$ is the sum of the dimensions of the
non-zero spaces $(\T_x \M)_{\chi}$ for which $\langle \lambda, \chi \rangle > 0$.
According to the comments at the beginning of this proof, $p(X)$ does not depend on the choice of $x \in X$.
These numbers are computed in Appendix \ref{appendix_A} using Singular \cite{singular} programs.
The Homology Basis Formula \cite[Theorem 4.4]{carrell}
\[
\H_m (\M, \ZZ) \isom \bigoplus_{X \in \Xi} \H_{m - 2p(X)} (X, \ZZ)
\]
implies that the homology groups $\H_m (\M, \ZZ)$, $0 \le m \le 52$, are torsion-free.
This also yields a formula for the Poincar\'e polynomial:
\[
P_{\M}(x) = \sum_{x \in \Pi} x^{2p(x)} + \sum_{X \in \Lambda} (x^2 + 1) x^{2p(X)} + \sum_{X \in \Sigma} (x^4 + 4x^2 +1) x^{2p(X)}.
\]
Starting from this, the Singular \cite{singular} program in Appendix \ref{appendix_A}
computes the expression for $P_{\M}$ given in the introduction.
Owing to the fact \cite[(4.7)]{carrell} that the Homology Basis Formula respects the Hodge decomposition,
we have the isomorphism
\[
\H^p (\M, \Omega_{\M}^q) \isom \bigoplus_{X \in \Xi} \H^{p - p(X)} (X, \Omega_X^{q - p(X)}).
\]
This shows that $h^{pq}(\M) = 0$ for $p \neq q$, the same property being true for the Hodge numbers of all $X \in \Xi$.

\appendix

\section{Singular programs}
\label{appendix_A}

\begin{verbatim}
ring r=0,(x,y,z),dp;

int i,j; poly P, q, q1, q2, l1, l2, l3; P=0;
int points, lines; points = 0; lines = 0;

list s1, s2, s2_0, s2_1, s2_2, s3, s4, s5, d;

s1=list(x,y,z);
s2=list(2x, 2y, 2z, x+y, x+z, y+z);
s2_0=list(2y, 2z, y+z);
s2_1=list(2x, 2z, x+z);
s2_2=list(2x, 2y, x+y);
s3=list(3x, 3y, 3z, 2x+y, 2x+z, x+2y, 2y+z, x+2z, y+2z, x+y+z);
s4=list(4x, 4y, 4z, 3x+y, 2x+2y, x+3y, 3x+z, 2x+2z, x+3z,
3y+z, 2y+2z, y+3z, 2x+y+z, x+2y+z, x+y+2z);
s5=list(5x,5y,5z,4x+y,3x+2y,2x+3y,x+4y,4x+z,3x+2z,2x+3z,x+4z,
4y+z,3y+2z,2y+3z,y+4z,3x+y+z,x+3y+z,x+y+3z,
2x+2y+z,2x+y+2z,x+2y+2z);

proc add(poly p, list l)
{int i; list ll; ll=list();
 for(i=1; i<=size(l); i=i+1){ll=ll+list(p+l[i]);};
 return(ll);};

proc positive_part(list l)
{int i; int p; p=0; for (i=1; i<=size(l); i=i+1) {if (l[i]>0) {p=p+1;};};
 return(p);};

proc values(list w, list l)
{int i; list v; v=list(); for (i=1; i<=size(w); i=i+1)
 {v=v+list((w[i]/x)*l[1]+(w[i]/y)*l[2]+(w[i]/z)*l[3]);};
 return(v);};

proc sub(list l, list ll)
{list lll; int i,j,e; lll=l; for (j=1; j<=size(ll); j=j+1)
 {e=1; for(i=1; i<=size(lll); i=i+1)
  {if (lll[i]==ll[j] and e==1) {lll=delete(lll,i); e=0;};};};
 return(lll);};
 
proc id2(list l)
{list ll; ll = list(); int i; for (i=1; i<=size(l); i=i+1)
 {ll = ll+ add(l[i], s3);};
 return(sub(s5, (sub(s5, ll))));};

proc id3(list l)
{list ll; ll = list(); int i; for (i=1; i<=size(l); i=i+1)
 {ll = ll+ add(l[i], s2);};
 return(sub(s5, (sub(s5, ll))));};

proc point_3(list l)
{points=points+3;
return(x^(2*positive_part(values(l, list(0,1,7))))
+x^(2*positive_part(values(l, list(7,1,0))))
+x^(2*positive_part(values(l, list(0,7,1)))));};

proc point_3_1(list l)
{points=points+3;
return(x^(2*positive_part(values(l, list(0,1,7))))
+x^(2*positive_part(values(l, list(1,0,7))))
+x^(2*positive_part(values(l, list(7,1,0)))));};

proc point_6(list l)
{points=points+6;
return(x^(2*positive_part(values(l, list(0,1,7))))
+x^(2*positive_part(values(l, list(1,0,7))))
+x^(2*positive_part(values(l, list(7,1,0))))
+x^(2*positive_part(values(l, list(1,7,0))))
+x^(2*positive_part(values(l, list(0,7,1))))
+x^(2*positive_part(values(l, list(7,0,1)))));};

proc line_3(list l)
{lines=lines+3;
return((1+x^2)*x^(2*positive_part(values(l, list(0,1,7))))
+(1+x^2)*x^(2*positive_part(values(l, list(7,1,0))))
+(1+x^2)*x^(2*positive_part(values(l, list(0,7,1)))));};

proc line_3_1(list l)
{lines=lines+3;
return((1+x^2)*x^(2*positive_part(values(l, list(0,1,7))))
+(1+x^2)*x^(2*positive_part(values(l, list(1,0,7))))
+(1+x^2)*x^(2*positive_part(values(l, list(7,1,0)))));};

proc line_6(list l)
{lines=lines+6;
return((1+x^2)*x^(2*positive_part(values(l, list(0,1,7))))
+(1+x^2)*x^(2*positive_part(values(l, list(1,0,7))))
+(1+x^2)*x^(2*positive_part(values(l, list(7,1,0))))
+(1+x^2)*x^(2*positive_part(values(l, list(1,7,0))))
+(1+x^2)*x^(2*positive_part(values(l, list(0,7,1))))
+(1+x^2)*x^(2*positive_part(values(l, list(7,0,1)))));};

proc surface_3(list l)
{return((1+4*(x^2)+x^4)*x^(2*positive_part(values(l, list(0,1,7))))
+(1+4*(x^2)+x^4)*x^(2*positive_part(values(l, list(7,1,0))))
+(1+4*(x^2)+x^4)*x^(2*positive_part(values(l, list(0,7,1)))));};

proc w0(list u, list v)
{return(add(-v[1]-u[1], s2) + add(-v[1]-u[2], s2) + add(-v[1]-u[3], s1)
+add(-v[2]-u[1], s2) + add(-v[2]-u[2], s2) + add(-v[2]-u[3], s1)
+add(-v[3]-u[1], s2) + add(-v[3]-u[2], s2) + add(-v[3]-u[3], s1));};

proc g0(list u, list v)
{return(list(u[1]-u[1], u[1]-u[2], -v[1]+v[1], -v[1]+v[2], -v[1]+v[3])
+ list(u[2]-u[1], u[2]-u[2], -v[2]+v[1], -v[2]+v[2], -v[2]+v[3])
+ add(u[3]-u[1], s1) +
add(u[3]-u[2], s1) + list(-v[3]+v[1], -v[3]+v[2], -v[3]+v[3]));};

proc m0(list u, list v)
{return(sub(w0(u,v), g0(u,v)));};

list alpha;
for (i=1; i<=3; i=i+1) {for (j=1; j<=3; j=j+1)
{alpha = m0(list(s2_0[i]-x, s2_1[j]-y, 0), list(x, y, z));
P=P+point_3(alpha);};};

alpha = m0(list(x, x, 0), list(x, y, z));
P=P+point_3_1(alpha);

list beta;
for (i=1; i<=3; i=i+1) {for (j=1; j<=3; j=j+1)
{beta = m0(list(s2_0[i]-x, s1[j], 0), list(x, y, z));
P=P+line_3_1(beta);};};

list gamma;
gamma = m0(list(x,y,0), list(x,y,z));
P = P + surface_3(gamma);

list delta;
q1=2x; q2=2y;
d=sub(id3(list(3x, x+2y, 2x+y, 3y)), list(2x+2y+z));
for (i=1; i<=size(d); i=i+1)
{delta= m0(list(d[i]-x-y-q2, d[i]-x-y-q1, 0),
list(x, y, -d[i]+q1+q2+x+y));
P = P + point_3(delta);};

d=list(2x+2y+z);
for (i=1; i<=size(d); i=i+1)
{delta= m0(list(d[i]-x-y-q2, d[i]-x-y-q1, 0),
list(x, y, -d[i]+q1+q2+x+y));
P = P + line_3(delta);};

q1=2x; q2=2z;
d=sub(id3(list(3x, x+2z, 2x+y, y+2z)), list(3x+y+z));
for (i=1; i<=size(d); i=i+1)
{delta= m0(list(d[i]-x-y-q2, d[i]-x-y-q1, 0),
list(x, y, -d[i]+q1+q2+x+y));
P = P + point_6(delta);};

q1=2z; q2=x+y;
d=sub(id3(list(x+2z, 2x+y, y+2z, x+2y)), list(2x+2y+z));
for (i=1; i<=size(d); i=i+1)
{delta= m0(list(d[i]-x-y-q2, d[i]-x-y-q1, 0), list(x, y, -d[i]+q1+q2+x+y));
P = P + point_3(delta);};

q1=2x; q2=y+z;
d=sub(id3(list(3x, x+y+z, 2x+y, 2y+z)), list(2x+y+2z, 3x+2y));
for (i=1; i<=size(d); i=i+1)
{delta= m0(list(d[i]-x-y-q2, d[i]-x-y-q1, 0), list(x, y, -d[i]+q1+q2+x+y));
P = P + point_6(delta);};

d=list(2x+y+2z);
for (i=1; i<=size(d); i=i+1)
{delta= m0(list(d[i]-x-y-q2, d[i]-x-y-q1, 0), list(x, y, -d[i]+q1+q2+x+y));
P = P + line_6(delta);};

q1=2x; q2=x+y;
d=sub(id3(list(3x, 2x+y, x+2y)), list(2x+y+2z, 3x+y+z, 2x+2y+z, 2x+3y));
for (i=1; i<=size(d); i=i+1)
{delta= m0(list(d[i]-x-y-q2, d[i]-x-y-q1, 0), list(x, y, -d[i]+q1+q2+x+y));
P = P + point_6(delta);};

d=list(2x+y+2z,3x+y+z,2x+2y+z,2x+3y);
for (i=1; i<=size(d); i=i+1)
{delta= m0(list(d[i]-x-y-q2, d[i]-x-y-q1, 0), list(x, y, -d[i]+q1+q2+x+y));
P = P + line_6(delta);};

q1=x+z; q2=y+z;
d=sub(id3(list(2x+z, x+y+z, 2y+z)), list(x+y+3z, x+3y+z, 2x+2y+z, 3x+y+z));
for (i=1; i<=size(d); i=i+1)
{delta= m0(list(d[i]-x-y-q2, d[i]-x-y-q1, 0), list(x, y, -d[i]+q1+q2+x+y));
P = P + point_3(delta);};

d=list(x+y+3z,x+3y+z,3x+y+z);
for (i=1; i<=size(d); i=i+1)
{delta= m0(list(d[i]-x-y-q2, d[i]-x-y-q1, 0), list(x, y, -d[i]+q1+q2+x+y));
P = P + line_3(delta);};

q1=2x; q2=x+z;
d=sub(id3(list(3x, 2x+z, 2x+y, x+y+z)),
list(3x+2z, 2x+y+2z, 2x+2y+z, 4x+y));
for (i=1; i<=size(d); i=i+1)
{delta= m0(list(d[i]-x-y-q2, d[i]-x-y-q1, 0),
list(x, y, -d[i]+q1+q2+x+y));
P = P + point_6(delta);};

d=list(3x+2z,2x+y+2z,2x+2y+z);
for (i=1; i<=size(d); i=i+1)
{delta= m0(list(d[i]-x-y-q2, d[i]-x-y-q1, 0),
list(x, y, -d[i]+q1+q2+x+y));
P = P + line_6(delta);};

q1=x+y; q2=y+z;
d=sub(id3(list(2x+y, x+y+z, x+2y, 2y+z)),
list(x+2y+2z, 2x+y+2z, 2x+3y, 3x+y+z));
for (i=1; i<=size(d); i=i+1)
{delta= m0(list(d[i]-x-y-q2, d[i]-x-y-q1, 0),
list(x, y, -d[i]+q1+q2+x+y));
P = P + point_6(delta);};

d=list(x+2y+2z,2x+y+2z,3x+y+z);
for (i=1; i<=size(d); i=i+1)
{delta= m0(list(d[i]-x-y-q2, d[i]-x-y-q1, 0),
list(x, y, -d[i]+q1+q2+x+y));
P = P + line_6(delta);};

q1=x+z; q2=2z;
d=sub(id3(list(2x+z, x+2z, x+y+z, y+2z)),
list(x+2y+2z, 2x+y+2z, 3x+2z));
for (i=1; i<=size(d); i=i+1)
{delta= m0(list(d[i]-x-y-q2, d[i]-x-y-q1, 0),
list(x, y, -d[i]+q1+q2+x+y));
P = P + point_6(delta);};

d=list(x+2y+2z, 3x+2z);
for (i=1; i<=size(d); i=i+1)
{delta= m0(list(d[i]-x-y-q2, d[i]-x-y-q1, 0),
list(x, y, -d[i]+q1+q2+x+y));
P = P + line_6(delta);};

proc w1(list u, list v)
{return(add(-v[1]-u[1],s1)+add(-v[1]-u[2],s1)+list(-v[1]-u[3],-v[1]-u[4])
+add(-v[2]-u[1],s2)+add(-v[2]-u[2],s2)+add(-v[2]-u[3],s1)+add(-v[2]-u[4],s1)
+add(-v[3]-u[1],s2)+add(-v[3]-u[2],s2)+add(-v[3]-u[3],s1)+add(-v[3]-u[4],s1)
+add(-v[4]-u[1],s2)+add(-v[4]-u[2],s2)+add(-v[4]-u[3],s1)+add(-v[4]-u[4],s1));};

proc g1(list u, list v)
{return(add(-v[2]+v[1], s1) + list(-v[2]+v[2], -v[2]+v[3], -v[2]+v[4])
+ add(-v[3]+v[1], s1) + list(-v[3]+v[2], -v[3]+v[3], -v[3]+v[4])
+ add(-v[4]+v[1], s1) + list(-v[4]+v[2], -v[4]+v[3], -v[4]+v[4])
+ list(0, 0, 0, 0, u[1]-u[2], u[2]-u[1], u[3]-u[4], u[4]-u[3])
+ add(u[3]-u[1], s1) + add(u[3]-u[2], s1)
+ add(u[4]-u[1], s1) + add(u[4]-u[2], s1));};

proc m1(list u, list v, list s)
{return(sub(w1(u,v), sub(g1(u,v), s)));};

list epsilon;
d=sub(id3(add(x, list(x+y, x+z, y+z)) +
add(y, list(x+y, x+z, y+z))), list(2x+2y+z, x+y+3z));
for (i=1; i<=size(d); i=i+1)
{epsilon = m1(list(d[i]-y, d[i]-x, x+y+z, x+y+z),
list(x+y-d[i], -y-z, -x-z, -x-y),
list(-d[i]+2x+2y+z, -d[i]+2x+2y+z));
P = P + point_3(epsilon);};

d=list(x+y+3z);
for (i=1; i<=size(d); i=i+1)
{epsilon = m1(list(d[i]-y, d[i]-x, x+y+z, x+y+z),
list(x+y-d[i], -y-z, -x-z, -x-y),
list(-d[i]+2x+2y+z, -d[i]+2x+2y+z));
P = P + line_3(epsilon);};

list zeta;
l1=x; l2=y;
d=sub(id3(add(x, list(2x, x+y, y+z)) +
add(y, list(2x, x+y, y+z))), list(3x+2y, 2x+2y+z, 2x+y+2z));
for (i=1; i<=size(d); i=i+1)
{zeta = m1(list(d[i]-l2, d[i]-l1, 2x+y, x+y+z),
list(-d[i]+l1+l2, -2x, -y-z, -x-y),
list(2x+y-d[i]+l1+l2, x+y+z-d[i]+l1+l2));
P = P + point_6(zeta);};

d=list(2x+y+2z);
for (i=1; i<=size(d); i=i+1)
{zeta = m1(list(d[i]-l2, d[i]-l1, 2x+y, x+y+z),
list(-d[i]+l1+l2, -2x, -y-z, -x-y),
list(2x+y-d[i]+l1+l2, x+y+z-d[i]+l1+l2));
P = P + line_6(zeta);};

l1=x; l2=z;
d=sub(id3(add(x, list(2x, x+y, y+z)) +
add(z, list(2x, x+y, y+z))), list(3x+y+z, 2x+y+2z, x+3y+z));
for (i=1; i<=size(d); i=i+1)
{zeta = m1(list(d[i]-l2, d[i]-l1, 2x+y, x+y+z),
list(-d[i]+l1+l2, -2x, -y-z, -x-y),
list(2x+y-d[i]+l1+l2, x+y+z-d[i]+l1+l2));
P = P + point_6(zeta);};

d=list(x+3y+z);
for (i=1; i<=size(d); i=i+1)
{zeta = m1(list(d[i]-l2, d[i]-l1, 2x+y, x+y+z),
list(-d[i]+l1+l2, -2x, -y-z, -x-y),
list(2x+y-d[i]+l1+l2, x+y+z-d[i]+l1+l2));
P = P + line_6(zeta);};

l1=y; l2=z;
d=sub(id3(add(y, list(2x, x+y, y+z)) +
add(z, list(2x, x+y, y+z))), list(2x+2y+z, x+2y+2z));
list s_z; s_z = list();
for (i=1; i<=size(d); i=i+1)
{zeta = m1(list(d[i]-l2, d[i]-l1, 2x+y, x+y+z),
list(-d[i]+l1+l2, -2x, -y-z, -x-y),
list(2x+y-d[i]+l1+l2, x+y+z-d[i]+l1+l2));
s_z= s_z + list(size(zeta));
P = P + point_6(zeta);};

list eta;
l1=x; l2=y;
d=sub(id3(add(x, list(2x, x+y, 2y)) + add(y, list(2x, x+y, 2y))),
list(3x+2y, 2x+3y, x+2y+2z, 2x+y+2z));
for (i=1; i<=size(d); i=i+1)
{eta= m1(list(d[i]-l2, d[i]-l1, x+2y, 2x+y),
list(-d[i]+l1+l2, -2y, -2x, -x-y),
list(x+2y-d[i]+l1+l2, 2x+y-d[i]+l1+l2));
P = P + point_3(eta);};

d=list(x+2y+2z, 2x+y+2z);
for (i=1; i<=size(d); i=i+1)
{eta= m1(list(d[i]-l2, d[i]-l1, x+2y, 2x+y),
list(-d[i]+l1+l2, -2y, -2x, -x-y),
list(x+2y-d[i]+l1+l2, 2x+y-d[i]+l1+l2));
P = P + line_3(eta);};

l1=x; l2=z;
d=sub(id3(add(x, list(2x, x+y, 2y)) + add(z, list(2x, x+y, 2y))),
list(2x+2y+z, 3x+y+z));
for (i=1; i<=size(d); i=i+1)
{eta= m1(list(d[i]-l2, d[i]-l1, x+2y, 2x+y),
list(-d[i]+l1+l2, -2y, -2x, -x-y),
list(x+2y-d[i]+l1+l2, 2x+y-d[i]+l1+l2));
P = P + point_6(eta);};

list theta;
l1=x; l2=y;
d=sub(id3(add(x, list(2x, x+y, x+z)) + add(y, list(2x, x+y, x+z))),
list(3x+y+z, 3x+2y, 2x+y+2z, x+3y+z, x+2y+2z));
for (i=1; i<=size(d); i=i+1)
{theta = m1(list(d[i]-l2, d[i]-l1, 2x+y, 2x+z),
list(-d[i]+l1+l2, -x-y, -x-z, -2x),
list(2x+y-d[i]+l1+l2, 2x+z-d[i]+l1+l2));
P = P + point_6(theta);};

d=list(2x+y+2z, x+3y+z, x+2y+2z);
for (i=1; i<=size(d); i=i+1)
{theta = m1(list(d[i]-l2, d[i]-l1, 2x+y, 2x+z),
list(-d[i]+l1+l2, -x-y, -x-z, -2x),
list(2x+y-d[i]+l1+l2, 2x+z-d[i]+l1+l2));
P = P + line_6(theta);};

l1=y; l2=z;
d=sub(id3(add(y, list(2x, x+y, x+z)) + add(z, list(2x, x+y, x+z))),
list(2x+2y+z, 2x+y+2z, x+y+3z, x+2y+2z, x+3y+z));
for (i=1; i<=size(d); i=i+1)
{theta = m1(list(d[i]-l2, d[i]-l1, 2x+y, 2x+z),
list(-d[i]+l1+l2, -x-y, -x-z, -2x),
list(2x+y-d[i]+l1+l2, 2x+z-d[i]+l1+l2));
P = P + point_3_1(theta);};

d=list(x+y+3z, x+2y+2z, x+3y+z);
list s_t; s_t = list();
for (i=1; i<=size(d); i=i+1)
{theta = m1(list(d[i]-l2, d[i]-l1, 2x+y, 2x+z),
list(-d[i]+l1+l2, -x-y, -x-z, -2x),
list(2x+y-d[i]+l1+l2, 2x+z-d[i]+l1+l2));
s_t = s_t + list(size(theta));
P = P + line_3_1(theta);};

proc w2(list u, list v)
{return(add(-v[1]-u[1], s1) + add(-v[1]-u[2], s1) + add(-v[1]-u[3], s1)
+add(-v[2]-u[1], s1) + add(-v[2]-u[2], s1) + add(-v[2]-u[3], s1)
+add(-v[3]-u[1], s3) + add(-v[3]-u[2], s3) + add(-v[3]-u[3], s3));};

proc g2(list u, list v)
{return(list(u[1]-u[1], u[1]-u[2], u[1]-u[3])
+ list(u[2]-u[1], u[2]-u[2], u[2]-u[3])
+ list(u[3]-u[1], u[3]-u[2], u[3]-u[3])
+ list(0, 0, -v[1]+v[2], -v[2]+v[1])
+ add(-v[3]+v[1], s2) + add(-v[3]+v[2], s2));};

proc m2(list u, list v)
{return(sub(w2(u,v), g2(u,v))
+ list(u[1]+v[3]-x-y-z, u[2]+v[3]-x-y-z, u[3]+v[3]-x-y-z));};

list iota;
d=sub(id2(list(x+y, x+z, y+z)), list(2x+2y+z, x+2y+2z, 2x+y+2z));
for (i=1; i<=size(d); i=i+1)
{iota = m2(list(x, y, z), list(0, 0, d[i]-x-y-z));
P = P + x^(2*positive_part(values(iota, list(0,1,7))));};

points=points+15;

list kappa;
d= sub(id2(list(2x, x+y, y+z)), list(x+2y+2z, 3x+y+z, 2x+2y+z));
for (i=1; i <=size(d); i=i+1)
{kappa = m2(list(y-x, x-z, 0), list(x, z, d[i]-x-y));
P = P + point_6(kappa);};

list lambda;
d= sub(id2(list(2x, x+y, 2y)), list(x+3y+z, 3x+y+z, 2x+2y+z));
for (i=1; i <=size(d); i=i+1)
{lambda = m2(list(x-y, y-x, 0), list(y, x, d[i]-x-y));
P = P + point_3(lambda);};

list mu;
d=sub(id2(list(2x, x+y, x+z)),
list(x+y+3z, x+2y+2z, x+3y+z, 2x+y+2z, 2x+2y+z, 3x+y+z));
for (i=1; i <=size(d); i=i+1)
{mu = m2(list(x-y, x-z, 0), list(y, z, d[i]-2x));
P = P + point_3_1(mu);};

d = list(x+y+3z, x+2y+2z, x+3y+z);
for (i=1; i <=size(d); i=i+1)
{mu = m2(list(x-y, x-z, 0), list(y, z, d[i]-2x));
P = P + line_3_1(mu);};

proc w3(list u, list v)
{return(add(-v[1]-u[1], s3) + add(-v[1]-u[2], s1)
+ add(-v[2]-u[1], s4) + add(-v[2]-u[2], s2));};

proc g3(list u, list v)
{return(list(0, 0, 0)
+ add(u[2]-u[1], s2) + add(-v[2]+v[1], s1));};

proc m3(list u, list v)
{return(sub(w3(u,v), g3(u,v))
+ list(u[1]+v[1]-x-y-z, u[1]+v[2]-2x-y-z, u[1]+v[2]-x-2y-z,
u[1]+v[2]-x-y-2z));};

list nu;
q=2y;
d= sub(s5, list(x+3y+z, 2x+y+2z, 3x+y+z, 2x+2y+z, 5z, y+4z));
for (i=1; i <=size(d); i=i+1)
{nu = m3(list(d[i]-q-x, 0), list(x, q));
P = P + point_6(nu);};

q=y+z;
d=sub(s5, list(x+2y+2z, 2x+y+2z, 3x+y+z, 2x+2y+z, 5y, 5z));
for (i=1; i <=size(d); i=i+1)
{nu = m3(list(d[i]-q-x, 0), list(x, q));
P = P + point_3_1(nu);};

\end{verbatim}

\end{document}